\documentclass{imsart}

\RequirePackage[OT1]{fontenc}
\RequirePackage{amsthm,amsmath}

\usepackage{amssymb}
\usepackage{amsmath}
\usepackage{amsfonts}
\usepackage{subfigure}
\usepackage{multirow}
\usepackage{xr}
\usepackage{float}
\usepackage{xspace}
\usepackage{algorithm}
\usepackage{algpseudocode}
\usepackage{algorithmicx}
\usepackage{color}
\usepackage{cite}
\usepackage{array}
\usepackage{color}
\usepackage{amsmath,amssymb}

\usepackage{natbib}
\usepackage{upgreek}
\usepackage{amsthm,amssymb,amsmath}
\usepackage{bm}
\usepackage{mathrsfs}
\usepackage{ctable}
\startlocaldefs
\numberwithin{equation}{section}
\theoremstyle{plain}
\newenvironment{remark}[1][Remark]{\begin{trivlist}
\item[\hskip \labelsep {\bfseries #1}]}{\end{trivlist}}
\newtheorem{theorem}{Theorem}
\newtheorem{lemma}{Lemma}
\newcommand{\Var}{\mathrm{Var}}
\newcommand{\E}{\mathrm{E}}
\endlocaldefs

\begin{document}

\begin{frontmatter}

% "Title of the paper"
\title{Bayesian two-step estimation in differential equation models}
\runtitle{Bayesian two-step estimation in differential equation models}

% indicate corresponding author with \corref{}
% \author{\fnms{John} \snm{Smith}\corref{}\ead[label=e1]{smith@foo.com}\thanksref{t1}}
% \thankstext{t1}{Thanks to somebody}
% \address{line 1\\ line 2\\ printead{e1}}
% \affiliation{Some University}

\author{\fnms{Prithwish} \snm{Bhaumik}\ead[label=e1]{pbhaumi@ncsu.edu}}
\address{Department of Statistics\\North Carolina State University\\SAS Hall, 2311 Stinson Drive\\Raleigh, NC 27695-8203\\\printead{e1}}
\and
\author{\fnms{Subhashis} \snm{Ghosal}\ead[label=e2]{sghosal@ncsu.edu}}
\address{Department of Statistics\\North Carolina State University\\4276 SAS Hall, 2311 Stinson Drive\\Raleigh, NC 27695-8203\\\printead{e2}}
\affiliation{North Carolina State University}

\runauthor{P. Bhaumik and S. Ghosal}

\begin{abstract}
Ordinary differential equations (ODEs) are used to model dynamic systems appearing in engineering, physics, biomedical sciences and many other fields. These equations contain unknown parameters, say $\bm\theta$ of physical significance which have to be estimated from the noisy data. Often there is no closed form analytic solution of the equations and hence we cannot use the usual non-linear least squares technique to estimate the unknown parameters. There is a two-step approach to solve this problem, where the first step involves fitting the data nonparametrically. In the second step the parameter is estimated by minimizing the distance between the nonparametrically estimated derivative and the derivative suggested by the system of ODEs. The statistical aspects of this approach have been studied under the frequentist framework. We consider this two-step estimation under the Bayesian framework. The response variable is allowed to be multidimensional and the true mean function of it is not assumed to be in the model. We induce a prior on the regression function using a random series based on the B-spline basis functions. We establish the Bernstein-von Mises theorem for the posterior distribution of the parameter of interest. Interestingly, even though the posterior distribution of the regression function based on splines converges at a rate slower than $n^{-1/2}$, the parameter vector $\bm\theta$ is nevertheless estimated at $n^{-1/2}$ rate.
\end{abstract}


\tableofcontents
\end{frontmatter}

\section{Introduction}

Suppose that we have a regression model $\bm Y=\bm f_{\bm\theta}(t)+\bm\varepsilon,\,\bm\theta\in\Theta\subseteq\mathbb{R}^p$. The explicit form of $\bm f_{\bm\theta}(\cdot)$ may not be known, but the function is assumed to satisfy the system of ordinary differential equations (ODEs) given by
\begin{eqnarray}
\frac{d\bm f_{\bm\theta}(t)}{dt}=\bm F(t,\bm f_{\bm\theta}(t),\bm\theta),\,t\in [0,1];\label{intro}
\end{eqnarray}
here $\bm F$ is a known appropriately smooth vector valued function and $\bm\theta$ is a parameter vector controlling the regression function. Equations of this type are encountered in various branches of science such as in genetics \citep{chen1999modeling}, viral dynamics of infectious diseases
[\citet{anderson1992infectious}, \citet{nowak2000virus}]. There are numerous applications in the fields of pharmacokinetics and pharmacodynamics (PKPD) as well. There are a lot of instances where no closed form solution exist.
Such an example can be found in the feedback system \citep[page 235]{gabrielsson2000pharmacokinetic} modeled by the ODEs
\begin{eqnarray}
\frac{dR(t)}{dt}&=&{k_{\mathrm{in}}}-k_{\mathrm{out}}R(t)(1+M(t)),\nonumber\\
\frac{dM(t)}{dt}&=&{k_{\mathrm{tol}}}(R(t)-M(t)),\nonumber
\end{eqnarray}
where $R(t)$ and $M(t)$ stand for loss of response and modulator at time $t$ respectively. Here ${k_{\mathrm{in}}}, {k_{\mathrm{out}}}$ and ${k_{\mathrm{ tol}}}$ are unknown parameters which have to be estimated from the noisy observations given by
\begin{eqnarray}
Y_R(t)&=&R(t)+\varepsilon_R(t),\nonumber\\
Y_M(t)&=&M(t)+\varepsilon_M(t),\nonumber
\end{eqnarray}
$\varepsilon_R(t), \varepsilon_M(t)$ being the respective noises at time point $t$. Another popular example is the Lotka-Volterra equations, also known as predator-prey equations. The prey and predator populations change over time according to the equations
\begin{eqnarray}
\frac{df_{1\bm\theta}(t)}{dt}&=&\theta_1f_{1\bm\theta}(t)-\theta_2f_{1\bm\theta}(t)f_{2\bm\theta}(t),\nonumber\\
\frac{df_{2\bm\theta}(t)}{dt}&=&-\theta_3f_{2\bm\theta}(t)+\theta_4f_{1\bm\theta}(t)f_{2\bm\theta}(t),\,\,t\in[0,1],\nonumber
\end{eqnarray}
where $\bm\theta=(\theta_1,\theta_2,\theta_3,\theta_4)^T$ and $f_{1\bm\theta}(t)$ and $f_{2\bm\theta}(t)$ denote the prey and predator populations at time $t$ respectively.\\
\indent If the ODEs can be solved analytically, then the usual non linear least squares (NLS) [\citet{levenberg1944method}, \citet{marquardt1963algorithm}] can be used to estimate the unknown parameters. In most of the practical situations, such closed form solutions are not available as evidenced in the previous two examples. NLS was modified for this purpose by \citet{bard1974nonlinear} and \citet{domselaar1975nonlinear}. \citet[page 134]{hairer1993solving} and \citet[page 53]{mattheij2002ordinary} used the 4-stage Runge-Kutta algorithm to solve \eqref{intro} numerically. The NLS can be applied in the next step to estimate the parameters. The statistical properties of the corresponding estimator have been studied by \citet{xue2010sieve}. The strong consistency, $\sqrt{n}$-consistency and asymptotic normality of the estimator were established in their work.\\
\indent \citet{ramsay2007parameter} proposed the generalized profiling procedure where the solution is approximated by a linear combination of basis functions. The coefficients of the basis functions are estimated by solving a penalized optimization problem using an initial choice of the parameters of interest. A data-dependent fitting criterion is constructed which contains the estimated coefficients. Then ${\bm\theta}$ is estimated by the maximizer of this criterion. \citet{qi2010asymptotic} explored the statistical properties of this estimator including $\sqrt{n}$-consistency and asymptotic normality. Despite having desirable statistical properties, these approaches are computationally cumbersome especially for high-dimensional systems of ODEs as well as when $\bm\theta$ is high-dimensional.\\
\indent \citet{varah1982spline} used an approach of two-step procedure. In the first step each of the state variables is approximated by a cubic spline using least squares technique. In the second step, the corresponding derivatives are estimated by differentiating the nonparametrically fitted curve and the parameter estimate is obtained by minimizing the sum of squares of difference between the derivatives of the fitted spline and the derivatives suggested by the ODEs at the design points. This method does not depend on the initial or boundary conditions of the state variables and is computationally quite efficient. An example given in \citet{voit2004decoupling} showed the computational superiority of the two-step approach over the usual least squares technique. \citet{brunel2008parameter} replaced the sum of squares of the second step by a weighted integral of the squared deviation and proved $\sqrt{n}$-consistency as well as asymptotic normality of the estimator so obtained. The order of the B-spline basis was determined by the smoothness of $\bm F(\cdot,\cdot,\cdot)$ with respect to its first two arguments. \citet{gugushvili2012n} used the same approach but used kernel smoothing instead of spline. They also established $\sqrt{n}$-consistency of the estimator. Another modification has been made in the work of \citet{wu2012numerical}. They have used penalized smoothing spline in the first step and numerical derivatives instead of actual derivatives of the nonparametrically estimated functions. In another work \citet{brunel2013parametric} used nonparametric approximation of the true solution to \eqref{intro} and then used a set of orthogonality conditions to estimate the parameters. The $\sqrt{n}-$consistency as well as the asymptotic normality of the estimator was also established in their work.\\
\indent In ODE models Bayesian estimation was considered in the works of \citet{gelman1996physiological}, \citet{rogers2007bayesian} and \citet{girolami2008bayesian}. First they solved the ODEs numerically to approximate the expected response and hence constructed the likelihood. A prior was assigned on $\bm\theta$ and MCMC technique was used to generate samples from the posterior. Computation cost might be an issue in this case as well. \citet{campbell2012smooth} proposed the smooth functional tempering approach which is a population MCMC technique and it utilizes the generalized profiling approach \citep{ramsay2007parameter} and the parallel tempering algorithm. \citet{jaeger2009functional} also used a Bayesian analog of the generalized profiling by putting prior on the coefficients of the basis functions. The theoretical aspects of Bayesian estimation methods have not been yet explored in the literature.\\
\indent In this paper we consider a Bayesian analog of the approach of \citet{brunel2008parameter} fitting a nonparametric regression model using B-spline basis. We assign priors on the coefficients of the basis functions. A posterior has been induced on $\bm\theta$ using the posteriors of the coefficients of the basis functions. In this paper we study the asymptotic properties of the posterior distribution of $\bm\theta$ and establish a Bernstein-von Mises theorem. We allow the ODE model to be misspecified, that is, the true regression function may not be a solution of the ODE. The response variable is also allowed to be multidimensional with possibly correlated errors. Normal distribution is used as the working model for error distribution, but the true distribution of errors may be different. Interestingly, the original model is parametric but it is embedded in a nonparametric model, which is further approximated by high dimensional parametric models. Note that the slow rate of nonparametric estimation does not influence the convergence rate of the parameter in the original parametric model.\\
\indent The paper is organized as follows. Section 2 contains the description of the notations and the model as well as the priors used for the analysis. The main results are given in Section 3. We extend the results to a more generalized set up in Section 4. In Section 5 we have carried out a simulation study under different settings. Proofs of the theorems are given in Section 6. Section 7 contains the proofs of the required lemmas.
\section{Notations, model assumption and prior specification}
We describe a set of notations to be used in this paper. Boldfaced letters are used to denote vectors and matrices. For a matrix $\bm A$, the symbols $\bm A_{i,}$ and $\bm A_{,j}$ stand for the $i^{th}$ row and $j^{th}$ column of $\bm A$ respectively. The notation $(\!(A_{i,j})\!)$ stands for a matrix with $(i,j)^{th}$ element being $A_{i,j}$. We use the notation $\mathrm{rows}_{r}^s(\bm A)$ with $r<s$ to denote the sub-matrix of $\bm A$ consisting of $r^{th}$ to $s^{th}$ rows of $\bm A$. Similarly, we can define $\mathrm{cols}_{r}^s(\bm A)$ for columns. The notation $\bm x_{r:s}$ stands for the sub-vector consisting of $r^{th}$ to $s^{th}$ elements of a vector $\bm x$. By $\mathrm{vec}(\bm A)$, we mean the vector obtained by stacking the columns of the matrix $\bm A$ one over another. For an $m\times n$ matrix $\bm A$ and a $p\times q$ matrix $\bm B$, $\bm A\otimes\bm B$ denotes the Kronecker product between $\bm A$ and $\bm B$; see \citet{steeb2006problems} for the definition. The identity matrix of order $p$ is denoted by $\bm I_p$. By the symbols $\mathrm{maxeig}(\bm A)$ and $\mathrm{mineig}(\bm A)$, we denote the maximum and minimum eigenvalues of the matrix $\bm A$ respectively.
For a vector $\bm x\in\mathbb{R}^p$, we denote $\|\bm x\|=\left(\sum_{i=1}^px_i^2\right)^{1/2}$. We denote the $r^{th}$ order derivative of a function $f(\cdot)$ by $f^{(r)}(\cdot)$, that is, $f^{(r)}(t)=\frac{d^r}{dt^r}f(t)$. The boldfaced symbol $\bm f(\cdot)$ stands for a vector valued function. For functions $\bm f:[0,1]\rightarrow\mathbb{R}^p$ and $w:[0,1]\rightarrow[0,\infty)$, we define $\|\bm f\|_w=(\int_0^1 \|\bm f(t)\|^2w(t)dt)^{1/2}$. For a real-valued function $f:[0,1]\rightarrow\mathbb{R}$ and a vector $\bm x\in\mathbb{R}^p$, we denote $f(\bm x)=(f(x_1),\ldots,f(x_p))^T$. The notation $\langle\cdot,\cdot\rangle$ stands for inner product. For numerical sequences $a_n$ and $b_n$, by $a_n=o(b_n)$, we mean $a_n/b_n\rightarrow0$ as $n\rightarrow\infty$. The notation $a_n=O(b_n)$ implies that $a_n/b_n$ is bounded. We use the notation $a_n\asymp b_n$ to mean $a_n=O(b_n)$ and $b_n=O(a_n)$, while $a_n\lesssim b_n$ stands for $a_n=O(b_n)$. The symbol $a_n\gg b_n$ will mean $b_n=o(a_n)$. Similarly we can define $a_n\ll b_n$. The notation $o_P(1)$ is used to indicate a sequence of random vectors which converges in probability to zero, whereas the expression $O_P(1)$ stands for a sequence of random vectors bounded in probability. The boldfaced symbols $\bm\E(\cdot)$ and $\bm\Var(\cdot)$ stand for the mean vector and dispersion matrix respectively of a random vector. For the probability measures $P$ and $Q$ defined on $\mathbb{R}^p$, we define the total variation distance $\|P-Q\|_{TV}=\sup_{B\in\mathscr{R}^p}|P(B)-Q(B)|$, where $\mathscr{R}^p$ denotes the Borel $\sigma$-field on $\mathbb R^p$. For an open set $E$, the symbol $C^m(E)$ stands for the collection of functions defined on $E$ with first $m$ continuous partial derivatives with respect to its arguments. Now let us consider the formal description of the model.\\
\indent We have a system of $d$ ordinary differential equations given by
\begin{eqnarray}
\frac{df_{j\bm\theta}(t)}{dt}=F_j(t,\bm f_{\bm\theta}(t),\bm\theta),\,t\in[0,1],\,j=1,\ldots,d,\label{diff}
\end{eqnarray}
where $\bm f_{\bm\theta}(\cdot)=(f_{1\bm\theta}(\cdot),\ldots,f_{d\bm\theta}(\cdot))^T$ and $\bm\theta\in\bm\Theta$, where we assume that $\bm\Theta$ is a compact subset of $\mathbb{R}^p$. Let us denote $\bm F(\cdot,\cdot,\cdot)=(F_1(\cdot,\cdot,\cdot),\ldots,F_d(\cdot,\cdot,\cdot))^T$. We also assume that for a fixed $\bm\theta$, $\bm F\in C^{m-1}((0,1),\mathbb{R}^d)$ for some integer $m\geq 1$. Then, by successive differentiation of the right hand side of \eqref{diff}, it follows that $\bm f_{\bm\theta}\in C^m((0,1))$. By the implied uniform continuity, the function and its several derivatives uniquely extend to continuous functions on $[0,1]$.\\
\indent Consider an $n\times d$ matrix of observations $\bm Y$ with $Y_{i,j}$ denoting the measurement taken on the $j^{th}$ response at the point $x_i,\,0\leq x_i\leq 1,\, i=1,\ldots,n;\,j=1,\ldots,d$. Denoting $\bm\varepsilon=(\!(\varepsilon_{i,j})\!)$ as the corresponding matrix of errors, the proposed model is given by
\begin{eqnarray}
Y_{i,j}=f_{j\bm\theta}(x_i)+\varepsilon_{i,j},\,i=1,\ldots,n,\,j=1,\ldots,d,\label{prop}
\end{eqnarray}
while the data is generated by the model
\begin{eqnarray}
Y_{i,j}=f_{j0}(x_i)+\varepsilon_{i,j},\,i=1,\ldots,n,\,j=1,\ldots,d,\label{true}
\end{eqnarray}
where $\bm f_0(\cdot)=(f_{10}(\cdot),\ldots,f_{d0}(\cdot))^T$ denotes the true mean vector which does not necessarily lie in $\{\bm f_{\bm\theta}:\bm\theta\in\bm\Theta\}$. We assume that $\bm f_0\in C^m([0,1])$. Let $\varepsilon_{i,j}\stackrel{iid}\sim P_0$, which is a probability distribution with mean zero and finite variance $\sigma_0^2$ for $i=1,\ldots,n\,;j=1,\ldots,d$.\\
Since the expression of $\bm f_{\bm\theta}$ is usually not available, the proposed model is embedded in nonparametric regression model
\begin{eqnarray}
\bm Y=\bm X_n\bm B_n+\bm\varepsilon,\label{np}
\end{eqnarray}
where $\bm X_n=(\!(N_{j}(x_i))\!)_{1\leq i\leq n,1\leq j\leq k_n+m-1}$, $\{N_{j}(\cdot)\}_{j=1}^{k_n+m-1}$ being the B-spline basis functions of order $m$ with $k_n-1$ interior knots. Here we denote $\bm B_n=\left(\bm\beta_1^{(k_n+m-1)\times 1},\ldots,\bm\beta_d^{(k_n+m-1)\times 1}\right)$, the matrix containing the coefficients of the basis functions. Also we consider $P_0$ to be unknown and use $N(0,\sigma^2)$ as the working distribution for the error where $\sigma$ may be treated as another unknown parameter. Let us denote by $t_1,\ldots,t_{k_n-1}$ the set of interior knots with $t_l=l/{k_n}$ for $l=1,\ldots,k_n-1$. Hence the meshwidth is $\xi_n={1}/{k_n}$. Denoting by $Q_n$, the empirical distribution function of $x_i,\,i=1,\ldots,n$, we assume
\begin{eqnarray}
\sup_{t\in[0,1]}|Q_n(t)-t|=o(k_n^{-1}).\nonumber
\end{eqnarray}
Let the prior distribution on the coefficients be given by
\begin{equation}
\bm\beta_j\stackrel{iid}\sim N_{k_n+m-1}(\bm 0,nk_n^{-1}(\bm X^T_n\bm X_n)^{-1}).\label{prior}
\end{equation}
Simple calculation yields the posterior distribution for $\bm\beta_j$ as
\begin{eqnarray}
\bm\beta_j|\bm Y\sim N_{k_n+m-1}\left(c^{-1}_n{({\bm X^T_n\bm X_n})}^{-1}{\bm X^T_n\bm Y_{,j}},c^{-1}_n\sigma^2(\bm X^T_n\bm X_n)^{-1}\right)\label{posterior}
\end{eqnarray}
and the posterior distributions of $\bm\beta_j$ and $\bm\beta_{j'}$ are mutually independent for $j\neq j';\,j,j'=1,\ldots,d$, where $c_n=(1+\sigma^2k_n/n)$. By model \eqref{np}, the expected response vector at a point $t\in[0,1]$ is given by $\bm B^T_n\bm N(t)$, where $\bm N(\cdot)=(N_{1}(\cdot),\ldots,N_{k_n+m-1}(\cdot))^T$.\\
\indent Let $w(\cdot)$ be a continuous weight function with $w(0)=w(1)=0$ and be positive on $(0,1)$. We define
\begin{eqnarray}
R_{\bm f}(\bm\eta)&=&\left\{\int_0^1\|\bm f'(t)-\bm F(t,\bm f(t),\bm\eta)\|^2w(t)dt\right\}^{1/2},\nonumber\\
\bm\psi(\bm f)&=&\arg\min_{\bm\eta\in\bm\Theta}R_{\bm f}(\bm\eta)\label{argmin}.
\end{eqnarray}
It is easy to check that $\bm\psi(\bm f_{\bm\eta})=\bm\eta$ for all $\bm\eta\in\Theta$. Thus the map $\bm\psi$ extends the definition of the parameter $\bm\theta$ beyond the model. Let us define $\bm\theta_0=\bm\psi(\bm f_0)$. We assume that $\bm\theta_0$ lies in the interior of $\bm\Theta$.
From now on, we shall write $\bm\theta$ for $\bm\psi(\bm f)$ and treat it as the parameter of interest. A posterior is induced on $\bm\Theta$ through the mapping $\bm\psi$ acting on $\bm f(\cdot)=\bm B^T_n\bm N(\cdot)$ and the posterior of $\bm B_n$ given by \eqref{posterior}.
\section{Main results}
Our objective is to study the asymptotic behavior of the posterior distribution of $\sqrt{n}(\bm\theta-\bm\theta_0)$. The asymptotic representation of $\sqrt{n}(\bm\theta-\bm\theta_0)$ is given by the next theorem under the assumption that
\begin{equation}
\text{for all}\,\epsilon>0,\,\,\inf_{\bm\eta:\|\bm\eta-\bm\theta_0\|\geq\epsilon}R_{\bm f_0}(\bm\eta)>R_{\bm f_0}(\bm\theta_0).\label{assmp}
\end{equation}
We denote $D_{l,r,s}\bm F(t,\bm f,\bm\theta)={\partial^{l+r+s}}/{\partial\bm\theta^s\partial\bm f^r\partial t^l}\bm F(t,\bm f(t),\bm\theta)$. Since the posterior distributions of $\bm\beta_j$ are mutually independent when $\bm\varepsilon_{,j}$ are mutually independent for $j=1,\ldots,d$, we can assume $d=1$ in Theorem 1 for the sake of simplicity in notation and write $f(\cdot)$, $f_0(\cdot)$, $F(\cdot,\cdot,\cdot)$, $\bm\beta$ instead of $\bm f(\cdot)$, $\bm f_0(\cdot)$, $\bm F(\cdot,\cdot,\cdot)$ and $\bm B_n$ respectively. Extension to $d$-dimensional case is straightforward as shown in Remark 4 after the statement of Theorem 1. We deal with the situation of correlated errors in Section 4.
\begin{theorem}\label{thm1}
Let the matrix
\begin{eqnarray}
\bm J_{\bm\theta_0}&=&\int_0^1(D_{0,0,1} F(t, f_0(t),\bm\theta_0))^TD_{0,0,1} F(t, f_0(t),\bm\theta_0)w(t)dt\nonumber\\
&&-\int_0^1\left(D_{0,0,1}\bm S(t, f_0(t),\bm\theta_0)\right)w(t)dt\nonumber
\end{eqnarray}
be nonsingular, where
\begin{eqnarray}
\bm S(t, f(t),\bm\theta)=(D_{0,0,1} F(t, f(t),\bm\theta))^T(f'_0(t)-F(t, f_0(t),\bm\theta_0)).\nonumber
\end{eqnarray}
Let $m$ be an integer greater than or equal to $5$ and $n^{1/2m}\ll k_n\ll n^{1/8}$. If $D_{0,2,1}F(t, y,\bm\theta)$ and $D_{0,0,2}F(t, y,\bm\theta)$ are continuous in their arguments, then under the assumption \eqref{assmp}, there exists $E_n\subseteq C^m((0,1))\times\bm\Theta$ with $\Pi(E^c_n|\bm Y)=o_{P_0}(1)$, such that uniformly for $(f, \bm\theta)\in E_n$,
\begin{eqnarray}
\|\sqrt{n}(\bm\theta-\bm\theta_0)-\bm{J}_{\bm\theta_0}^{-1}\sqrt{n}(\bm\Gamma(f)-\bm\Gamma(f_{0}))\|\rightarrow 0\label{thm1}
\end{eqnarray}
as $n\rightarrow\infty$, where
\begin{align}
\bm\Gamma(z)=&\int_0^1 \left(-(D_{0,0,1} F(t, f_0(t),\bm\theta_0))^TD_{0,1,0} F(t, f_0(t),\bm\theta_0)w(t)\right.\nonumber\\
&\left.-\frac{d}{dt}[(D_{0,0,1} F(t, f_0(t),\bm\theta_0))^Tw(t)]+\left(D_{0,1,0}\bm S(t, f_0(t),\bm\theta_0)\right)w(t)\right)z(t)dt.\nonumber
\end{align}
\end{theorem}
\begin{remark}[Remark 1:]
Condition \eqref{assmp} implies that $\bm\theta_0$ is the unique point of minimum of $R_{\bm f_0}(\cdot)$ and $\bm\theta_0$ should be a well-separated point of minimum.
\end{remark}
\begin{remark}[Remark 2:]
The posterior distribution of $\bm\Gamma(\bm f)-\bm\Gamma(\bm f_0)$ contracts at $\bm 0$ at the rate $n^{-1/2}$ as indicated by Lemma 4. Hence, the posterior distribution of $(\bm\theta-\bm\theta_0)$ contracts at $\bm 0$ at the rate $n^{-1/2}$ with high probability under the truth. We refer to Theorem 2 for a more refined version of this result.
\end{remark}
\begin{remark}[Remark 3:]
We note that a minimum of fifth order smoothness of the true mean function is good enough to ensure the contraction rate $n^{-1/2}$. We do not gain anything more by assuming a higher order of smoothness. For $m=5$, the required condition becomes $n^{1/10}\ll k_n\ll n^{1/8}$. Also, the knots are chosen deterministically and there is no need to assign a prior on the number of terms of the random series used. Hence, the issue of Bayesian adaptation, that is, improving convergence rate with higher smoothness without knowing the smoothness, does not arise in the present context.
\end{remark}
\begin{remark}[Remark 4:]
When the response is a $d$-dimensional vector, \eqref{thm1} holds with the scalars being replaced by the corresponding $d$-dimensional vectors. Let $\bm A(t)$ stands for the $p\times d$ matrix
\begin{eqnarray}
\lefteqn{\bm J_{\bm\theta_0}^{-1}\{-(D_{0,0,1}\bm F(t,\bm f_0(t),\bm\theta_0))^TD_{0,1,0}\bm F(t,\bm f_0(t),\bm\theta_0)w(t)}\nonumber\\
&&-\frac{d}{dt}[(D_{0,0,1}\bm F(t,\bm f_0(t),\bm\theta_0))^Tw(t)]\nonumber\\
&&+\left(D_{0,1,0}\bm S(t,\bm f_0(t),\bm\theta_0)\right)w(t)\}.\nonumber
\end{eqnarray}
Then we have
\begin{eqnarray}
\bm J_{\bm\theta_0}^{-1}\bm\Gamma(\bm f)
=\sum_{j=1}^d\int_0^1\bm A_{,j}(t)\bm N^T(t)\bm\beta_jdt
=\sum_{j=1}^d\bm G_{n,j}^T\bm\beta_j,\label{linearize}
\end{eqnarray}
where $\bm G_{n,j}^T=\int_0^1\bm A_{,j}(t)\bm N^T(t)dt$ which is a $p\times (k_n+m-1)$ matrix for $j=1,\ldots,d$. Then in order to approximate the posterior distribution of $\bm\theta$, it suffices to study the asymptotic posterior distribution of the linear combination of $\bm\beta_j$ given by \eqref{linearize}. The next theorem describes the approximate posterior distribution of $\sqrt{n}(\bm\theta-\bm\theta_0)$.
\end{remark}

\begin{theorem}
Define
\begin{eqnarray}
\bm\mu_n&=&\sqrt{n}\sum_{j=1}^d\bm G_{n,j}^T(\bm X^T_n\bm X_n)^{-1}\bm X^T_n\bm Y_{,j}-\sqrt{n}\bm J_{\bm\theta_0}^{-1}\bm\Gamma(\bm f_0),\nonumber\\
\bm\Sigma_n&=&n\sum_{j=1}^d\bm G_{n,j}^T(\bm X^T_n\bm X_n)^{-1}\bm G_{n,j}\nonumber
\end{eqnarray}
and $\bm B_j=\left(\!\left(\left<A_{k,j}(\cdot),A_{k',j}(\cdot)\right>\right)\!\right)_{k,k'=1,\ldots,p}$ for $j=1,\ldots,d$. If $\bm B_j$ is non-singular for all $j=1,\ldots,d$, then under the conditions of Theorem 1,
\begin{eqnarray}
\left\|\Pi\left(\sqrt{n}(\bm\theta-\bm\theta_0)\in\cdot|\bm Y\right)-N\left(\bm\mu_n,\sigma^2\bm\Sigma_n\right)\right\|_{TV}=o_{P_0}(1).\label{thm2}
\end{eqnarray}
\end{theorem}

Inspecting the proof, we can conclude that \eqref{thm2} is uniform over $\sigma^2$ belonging to a compact subset of $(0,\infty)$. Also note that the scale of the approximating normal distribution involves the working variance $\sigma^2$ assuming that it is given, even though the convergence is studied under the true distribution $P_0$ with variance $\sigma_0^2$, not necessarily equal to the given $\sigma^2$. Thus, the distribution matches with the frequentist distribution of the estimator in \citet{brunel2008parameter} only if $\sigma$ is correctly specified as $\sigma_0$. The next result assures that putting a prior on $\sigma$ rectifies the problem.\\

\begin{theorem}
We assign independent $N(\bm 0, nk^{-1}_n\sigma^2(\bm X_n^TX_n)^{-1})$ prior on $\bm\beta_j$ for $j=1,\ldots,d$, and inverse gamma prior on $\sigma^2$ with shape and scale parameters $a$ and $b$ respectively. If the fourth order moment of the true error distribution is finite, then
\begin{eqnarray}
\left\|\Pi\left(\sqrt{n}(\bm\theta-\bm\theta_0)\in\cdot|\bm Y\right)-N\left(\bm\mu_n,\sigma^2_0\bm\Sigma_n\right)\right\|_{TV}=o_{P_0}(1).\label{thm3}
\end{eqnarray}
 \end{theorem}
\section{Extension}
The results obtained so far can be extended for the case where $\bm\varepsilon_{i,j}$ and $\bm\varepsilon_{i,j'}$ are associated for $i=1,\ldots,n$ and $j\neq j';\,j,j'=1,\ldots,d$. Let under the working model, $\bm\varepsilon_{i,}$ have the dispersion matrix $\bm\Sigma=\sigma^2\bm\Omega$ for $i=1,\ldots,n$, $\bm\Omega$ being a known positive definite matrix. Denoting $\bm\Omega^{-1/2}=(\!(\omega^{jk})\!)_{j,k=1}^d$, we have the following extension of Theorem 2.

\begin{theorem}
Define
\begin{eqnarray}
\bm\mu^*_n&=&\sqrt{n}\sum_{k=1}^d\mathrm{cols}_{(k-1)(k_n+m-1)+1}^{k(k_n+m-1)}\left(\left(\bm G^T_{n,1}\ldots\bm G^T_{n,d}\right)\left(\bm\Omega^{1/2}\otimes\bm I_{k_n+m-1}\right)\right)\nonumber\\
&&\times{(\bm X^T_n\bm X_n)}^{-1}\bm X_n^T\sum_{j=1}^d\bm Y_{,j}\omega^{jk}
-\sqrt{n}\bm J_{\bm\theta_0}^{-1}\bm\Gamma(\bm f_0),\nonumber\\
\bm\Sigma^*_n&=&n\sum_{k=1}^d\mathrm{cols}_{(k-1)(k_n+m-1)+1}^{k(k_n+m-1)}\left(\left(\bm G^T_{n,1}\ldots\bm G^T_{n,d}\right)\left(\bm\Omega^{1/2}\otimes\bm I_{k_n+m-1}\right)\right)\nonumber\\
&&\times{(\bm X^T_n\bm X_n)}^{-1}\nonumber\\
&&\times\mathrm{rows}_{(k-1)(k_n+m-1)+1}^{k(k_n+m-1)}\left(\left(\bm\Omega^{1/2}\otimes\bm I_{k_n+m-1}\right)\left(\bm G^T_{n,1}\ldots\bm G^T_{n,d}\right)^T\right).\nonumber
\end{eqnarray}
Then under the conditions of Theorem 1 and Theorem 2,
\begin{eqnarray}
\left\|\Pi\left(\sqrt{n}(\bm\theta-\bm\theta_0)\in\cdot|\bm Y\right)-N\left(\bm\mu^*_n,\sigma^2\bm\Sigma^*_n\right)\right\|_{TV}=o_{P_0}(1).\label{thm4_1}
\end{eqnarray}
If $\sigma^2$ is unknown and is given an inverse gamma prior, then
\begin{eqnarray}
\left\|\Pi\left(\sqrt{n}(\bm\theta-\bm\theta_0)\in\cdot|\bm Y\right)-N\left(\bm\mu^*_n,\sigma^2_0\bm\Sigma^*_n\right)\right\|_{TV}=o_{P_0}(1),\label{thm4_2}
\end{eqnarray}
where $\sigma^2_0$ is the true value of $\sigma^2$.
\end{theorem}

\begin{remark}[Remark 5:]
In many applications, the regression function is modeled as $\bm h_{\bm\theta}(t)=\bm g(\bm f_{\bm\theta}(t))$ instead of $\bm f_{\bm\theta}(t)$, where $\bm g$ is a known invertible function and $\bm h_{\bm\theta}(t)\in\mathbb{R}^d$. It should be noted that
\begin{eqnarray}
\frac{d\bm h_{\bm\theta}(t)}{dt}=\bm g'(\bm f_{\bm\theta}(t))\frac{d\bm f_{\bm\theta}(t)}{dt}&=&\bm g'(\bm g^{-1}\bm h_{\bm\theta}(t))\bm F(t,\bm g^{-1}\bm h_{\bm\theta}(t),\bm\theta)\nonumber\\
&=&\bm H(t,\bm h_{\bm\theta}(t),\bm\theta),\nonumber
\end{eqnarray}
which is known function of $t, \bm h_{\bm\theta}$ and $\bm\theta$. Now we can carry out our analysis replacing $\bm F$ and $\bm f_{\bm\theta}$ in \eqref{intro} by $\bm H$ and $\bm h_{\bm\theta}$ respectively.
\end{remark}
\section{Simulation Study}
We consider the Lotka-Volterra equations to study the posterior distribution of $\bm\theta$. We consider both the situations where the true regression function belongs to the solution set and lies outside the solution set. For a sample of size $n$, the $x_i$'s are chosen as $x_i=(2i-1)/2n$ for $i=1,\ldots,n$. Samples of sizes 50, 100 and 500 are considered. The weight function is chosen as $w(t)=t(1-t),\,t\in[0,1]$. We simulate 1000 replications for each case. Under each replication a sample of size 1000 is directly drawn from the posterior distribution of $\bm\theta$ and then 95\% equal tailed credible interval is obtained. Each replication took around one minute. We calculate the coverage and the average length of the corresponding credible interval over these 1000 replications. The estimated standard errors of the interval length and coverage are given inside the parentheses in the tables. We also consider 1000 replications to construct the $95\%$ equal tailed confidence interval based on asymptotic normality as obtained from the estimation method introduced by \citet{varah1982spline} and modified and studied by \citet{brunel2008parameter}. We abbreviate this method by ``VB" in tables. The estimated standard errors of the interval length and coverage are given inside the parentheses in the tables.
Thus we have $p=4,d=2$ and the ODE's are given by
\begin{eqnarray}
F_1(t, \bm f_{\bm\theta}(t), \bm\theta)&=&\theta_1f_{1\bm\theta}(t)-\theta_2f_{1\bm\theta}(t)f_{2\bm\theta}(t),\nonumber\\
F_2(t, \bm f_{\bm\theta}(t), \bm\theta)&=&-\theta_3f_{2\bm\theta}(t)+\theta_4f_{1\bm\theta}(t)f_{2\bm\theta}(t),\,\,t\in[0,1],\nonumber
\end{eqnarray}
with initial condition $f_{1\bm\theta}(0)=1,f_{2\bm\theta}(0)=0.5$.
The above system is not analytically solvable.
We consider two cases.\\
\indent\textbf{\it{Case 1}:}
The true regression function is in the model. Thus the true mean vector is given by $\left(f_{1\bm\theta_0}(t), f_{2\bm\theta_0}(t)\right)^T$, where $\bm\theta_0=(\theta_{10},\theta_{20},\theta_{30},\theta_{40})^T$. We take $\theta_{10}=\theta_{20}=\theta_{30}=\theta_{40}=10$ to be the true value of the parameter.\\
\indent\textbf{\it{Case 2}:}
The true mean vector is taken outside the solution set of the ODE and is given by $\left(f_{1\bm\theta_0}(t)+0.4\sin(4\pi t), f_{2\bm\theta_0}(t)+0.4\sin(4\pi t)\right)^T$.\\

%\begin{eqnarray}
%f_{10}(t)&=&f_{1\bm\theta}(t)+0.4\sin(4\pi t),\nonumber\\
%f_{20}(t)&=&f_{2\bm\theta}(t)+0.4\sin(4\pi t).\nonumber
%\end{eqnarray}
\indent The true distribution of error is taken either $N(0,(0.2)^2)$ or a scaled $t$-distribution with $6$ degrees of freedom, where scaling is done in order to make the standard deviation $0.2$. We put an inverse gamma prior on $\sigma^2$ with shape and scale parameters being $99$ and $1$ respectively and independent Gaussian priors on $\bm\beta_1$ and $\bm\beta_2$ with mean vector $\bm 0$ and dispersion matrix $nk^{-1}_n\sigma^2(\bm X^T_n\bm X_n)^{-1}$. As far as choosing $k_n$ is concerned, we take $k_n$ in the order of $n^{1/9}$. The simulation results are summarized in the tables 1 and 2. Not surprisingly asymptotic normality based confidence intervals obtained from VB method are shorter but too optimistic, failing to give adequate coverage for finite sample sizes. On the other hand, the posterior credible intervals appear to be slightly conservative.

\begin{table}[h]
\centering
\caption{\textit{Coverages and average lengths of the Bayesian credible interval and confidence interval obtained from VB method for Case 1}}
\begin{tabular}{|c|c|cc|cc|cc|cc|}
\hline
\multirow{2}{*}{$n$}&&\multicolumn{4}{|c|}{$N(0,(0.2)^2)$}&\multicolumn{4}{|c|}{scaled $t_6$}\\
\cline{3-10}
&&\multicolumn{2}{|c|}{Bayes}&\multicolumn{2}{|c|}{VB}&\multicolumn{2}{|c|}{Bayes}&\multicolumn{2}{|c|}{VB}\\
\hline
&&coverage&length&coverage&length&coverage&length&coverage&length\\
&&(se)&(se)&(se)&(se)&(se)&(se)&(se)&(se)\\
50&$\theta_1$&97.7&9.93&83.6&4.57&97.8&9.91&83.6&4.55\\
&&(0.02)&(1.04)&(0.05)&(0.84)&(0.02)&(1.02)&(0.05)&(0.93)\\
&$\theta_2$&100.0&12.55&82.4&4.26&100.0&12.48&81.1&4.24\\
&&(0.00)&(1.27)&(0.05)&(0.80)&(0.00)&(1.24)&(0.06)&(0.87)\\
&$\theta_3$&99.6&10.68&86.6&4.96&99.0&10.65&85.1&4.93\\
&&(0.01)&(1.21)&(0.05)&(0.98)&(0.01)&(1.23)&(0.05)&(1.06)\\
&$\theta_4$&100.0&13.00&85.5&4.39&100.0&12.94&84.3&4.36\\
&&(0.00)&(1.40)&(0.05)&(0.88)&(0.00)&(1.42)&(0.05)&(0.94)\\
\cline{2-10}
100&$\theta_1$&98.9&6.67&88.6&3.38&98.6&6.63&87.6&3.34\\
&&(0.01)&(0.58)&(0.03)&(0.46)&(0.01)&(0.57)&(0.03)&(0.47)\\
&$\theta_2$&100.0&7.41&89.0&3.15&100.0&7.37&88.1&3.13\\
&&(0.00)&(0.63)&(0.03)&(0.44)&(0.00)&(0.62)&(0.03)&(0.45)\\
&$\theta_3$&99.3&7.08&88.9&3.60&99.4&7.12&89.7&3.61\\
&&(0.01)&(0.62)&(0.03)&(0.51)&(0.01)&(0.65)&(0.03)&(0.56)\\
&$\theta_4$&100.0&7.54&87.4&3.19&100.0&7.59&89.4&3.21\\
&&(0.00)&(0.62)&(0.03)&(0.46)&(0.00)&(0.67)&(0.03)&(0.50)\\
\cline{2-10}
500&$\theta_1$&98.3&2.24&94.6&1.55&98.5&2.25&93.8&1.55\\
&&(0.00)&(0.11)&(0.01)&(0.09)&(0.00)&(0.12)&(0.01)&(0.10)\\
&$\theta_2$&99.6&2.20&93.8&1.45&99.5&2.21&93.2&1.45\\
&&(0.00)&(0.10)&(0.01)&(0.09)&(0.00)&(0.12)&(0.01)&(0.10)\\
&$\theta_3$&99.2&2.41&93.8&1.66&98.6&2.40&93.8&1.66\\
&&(0.00)&(0.13)&(0.01)&(0.10)&(0.00)&(0.13)&(0.01)&(0.12)\\
&$\theta_4$&99.6&2.24&93.5&1.48&99.6&2.24&94.3&1.48\\
&&(0.00)&(0.11)&(0.01)&(0.09)&(0.00)&(0.12)&(0.01)&(0.10)\\
\hline
\end{tabular}
\end{table}

\begin{table}[h]
\centering
\caption{\textit{Coverages and average lengths of the Bayesian credible interval and confidence interval obtained from VB method for Case 2}}
\begin{tabular}{|c|c|cc|cc|cc|cc|}
\hline
\multirow{2}{*}{$n$}&&\multicolumn{4}{|c|}{$N(0,(0.2)^2)$}&\multicolumn{4}{|c|}{scaled $t_6$}\\
\cline{3-10}
&&\multicolumn{2}{|c|}{Bayes}&\multicolumn{2}{|c|}{VB}&\multicolumn{2}{|c|}{Bayes}&\multicolumn{2}{|c|}{VB}\\
\hline
&&coverage&length&coverage&length&coverage&length&coverage&length\\
&&(se)&(se)&(se)&(se)&(se)&(se)&(se)&(se)\\
50&$\theta_1$&98.2&13.62&84.9&7.28&98.3&13.57&85.1&7.24\\
&&(0.02)&(1.94)&(0.05)&(1.84)&(0.02)&(1.90)&(0.05)&(1.90)\\
&$\theta_2$&100.0&17.03&84.6&6.82&100.0&16.94&84.0&6.77\\
&&(0.00)&(2.40)&(0.05)&(1.76)&(0.00)&(2.33)&(0.05)&(1.81)\\
&$\theta_3$&100.0&7.60&87.3&2.79&100.0&7.59&87.8&2.78\\
&&(0.00)&(0.80)&(0.05)&(0.60)&(0.00)&(0.81)&(0.05)&(0.67)\\
&$\theta_4$&100.0&8.91&83.7&2.28&100.0&8.89&84.0&2.27\\
&&(0.00)&(0.83)&(0.05)&(0.50)&(0.00)&(0.85)&(0.05)&(0.55)\\
\cline{2-10}
100&$\theta_1$&99.0&9.77&89.8&5.27&99.2&9.75&90.9&5.24\\
&&(0.01)&(1.24)&(0.03)&(0.95)&(0.01)&(1.20)&(0.03)&(0.95)\\
&$\theta_2$&100.0&10.82&89.0&4.94&100.0&10.82&90.7&4.92\\
&&(0.00)&(1.36)&(0.03)&(0.91)&(0.00)&(1.33)&(0.03)&(0.92)\\
&$\theta_3$&99.9&4.76&85.2&1.90&100.0&4.76&86.5&1.90\\
&&(0.00)&(0.39)&(0.04)&(0.30)&(0.00)&(0.41)&(0.03)&(0.34)\\
&$\theta_4$&100.0&4.96&81.5&1.54&100.0&4.96&84.5&1.54\\
&&(0.00)&(0.36)&(0.04)&(0.25)&(0.00)&(0.38)&(0.04)&(0.28)\\
\cline{2-10}
500&$\theta_1$&98.6&3.53&92.6&2.36&98.7&3.56&92.8&2.37\\
&&(0.00)&(0.26)&(0.01)&(0.19)&(0.00)&(0.27)&(0.01)&(0.21)\\
&$\theta_2$&99.6&3.48&92.2&2.22&99.7&3.51&93.3&2.23\\
&&(0.00)&(0.26)&(0.01)&(0.18)&(0.00)&(0.26)&(0.01)&(0.20)\\
&$\theta_3$&99.5&1.53&89.2&0.83&99.2&1.53&86.3&0.83\\
&&(0.00)&(0.07)&(0.01)&(0.06)&(0.00)&(0.07)&(0.02)&(0.06)\\
&$\theta_4$&99.2&1.41&85.5&0.67&99.3&1.41&84.3&0.67\\
&&(0.00)&(0.06)&(0.02)&(0.05)&(0.00)&(0.06)&(0.02)&(0.05)\\
\hline
\end{tabular}
\end{table}

\section{Proofs}
\begin{proof}[Proof of Theorem 1]
The structure of the proof follows that of Proposition 3.1 of \citet{brunel2008parameter} and Proposition 3.3 of \citet{gugushvili2012n}, but differs substantially in detail since we address posterior variation and also allow misspecification. First note that
\begin{eqnarray}
\lefteqn{\bm\Gamma(f)-\bm\Gamma(f_0)=\int_0^1 \left(-(D_{0,0,1} F(t, f_0(t),\bm\theta_0))^TD_{0,1,0} F(t, f_0(t),\bm\theta_0)w(t)\right.}\label{lin}\\
&&\left.-\frac{d}{dt}[(D_{0,0,1} F(t, f_0(t),\bm\theta_0))^Tw(t)]+\left(D_{0,1,0}\bm S(t, f_0(t),\bm\theta_0)\right)w(t)\right)(f(t)-f_0(t))dt.\nonumber
\end{eqnarray}
Interchanging the orders of differentiation and integration and using the definitions of $\bm\theta$ and $\bm\theta_0$,
\begin{eqnarray}
\int_0^1\left(D_{0,0,1}F(t, f(t),\bm\theta)\right)^T(f'(t)-F(t, f(t),\bm\theta))w(t)dt&=&\bm0,\label{theta_1}\\
\int_0^1\left(D_{0,0,1}F(t, f_0(t),\bm\theta_0)\right)^T(f'_0(t)-F(t, f_0(t),\bm\theta_0))w(t)dt&=&\bm0.\label{theta0}
\end{eqnarray}
Taking difference, we get
\begin{eqnarray}
\lefteqn{\int_0^1\left((D_{0,0,1}F(t, f(t),\bm\theta)-D_{0,0,1}F(t, f(t),\bm\theta_0))^T(f'_0(t)-F(t, f_0(t),\bm\theta_0))\right)w(t)dt}\nonumber\\
&&+\int_0^1(D_{0,0,1}F(t, f(t),\bm\theta_0)-D_{0,0,1}F(t, f_0(t),\bm\theta_0))^T(f'_0(t)-F(t, f_0(t),\bm\theta_0))w(t)dt\nonumber\\
&&+\int_0^1\left(D_{0,0,1}F(t, f(t),\bm\theta_0)\right)^T(f'(t)-f'_0(t)+F(t, f_0(t),\bm\theta_0)-F(t, f(t),\bm\theta_0))w(t)dt\nonumber\\
&&+\int_0^1\left(D_{0,0,1}F(t, f(t),\bm\theta)-D_{0,0,1}F(t, f(t),\bm\theta_0)\right)^T(f'(t)-f'_0(t)\nonumber\\
&&+F(t, f_0(t),\bm\theta_0)-F(t, f(t),\bm\theta_0))w(t)dt\nonumber\\
&&+\int_0^1\left(D_{0,0,1}F(t, f(t),\bm\theta)\right)^T\left(F(t, f(t),\bm\theta_0)-F(t, f(t),\bm\theta)\right)w(t)dt=\bm0.\nonumber
\end{eqnarray}
Replacing the difference between the values of a function at two different values of an argument by the integral of the corresponding partial derivative, we get
\begin{eqnarray}
\lefteqn{\bm M(f,\bm\theta)(\bm\theta-\bm\theta_0)}\nonumber\\
&=&\int_0^1\left(D_{0,0,1}F(t, f(t),\bm\theta_0)-D_{0,0,1}F(t, f_0(t),\bm\theta_0)\right)^T(f'_0(t)-F(t, f_0(t),\bm\theta_0))w(t)dt\nonumber\\
&&+\int_0^1\left(D_{0,0,1}F(t, f(t),\bm\theta_0)\right)^T(f'(t)-f'_0(t)+F(t, f_0(t),\bm\theta_0)-F(t, f(t),\bm\theta_0))w(t)dt,\nonumber
\end{eqnarray}
where $\bm M(f,\bm\theta)$ is given by
\begin{eqnarray}
\lefteqn{\int_0^1\left(D_{0,0,1}F(t, f(t),\bm\theta)\right)^T\left\{\int_0^1D_{0,0,1}F(t, f(t),\bm\theta_0+\lambda(\bm\theta-\bm\theta_0))d\lambda\right\} w(t)dt}\nonumber\\
&&-\int_0^1\left\{\int_0^1\left(D_{0,0,1}\bm S(t, f(t),\bm\theta_0+\lambda(\bm\theta-\bm\theta_0))\right)d\lambda\right\}w(t)dt\nonumber\\
&&-\int_0^1\left\{\int_0^1\left(D_{0,0,2}\bm F(t, f(t),\bm\theta_0+\lambda(\bm\theta-\bm\theta_0))\right)d\lambda\right\}(f'(t)-f'_0(t)\nonumber\\
&&+F(t, f_0(t),\bm\theta_0)-F(t, f(t),\bm\theta_0))w(t)dt\nonumber.
\end{eqnarray}
Note that $\bm M(f_0,\bm\theta_0)=\bm J_{\bm\theta_0}$. We also define
\begin{eqnarray}
E_n=\{(f,\bm\theta):\sup_{t\in[0,1]}|f(t)-f_{0}(t)|\leq\epsilon_n, \|\bm\theta-\bm\theta_0\|\leq\epsilon_n\},\nonumber
\end{eqnarray}
where $\epsilon_n\rightarrow0$. By Lemmas 2 and 3, there exists such a sequence $\{\epsilon_n\}$ so that $\Pi(E^c_n|\bm Y)=o_{P_0}(1)$. Then, $\bm M(f,\bm\theta)$ is invertible and the eigenvalues of $[\bm M(f,\bm\theta)]^{-1}$ are bounded away from $0$ and $\infty$ for sufficiently large $n$ and $\|(\bm M(f,\bm\theta))^{-1}-\bm J^{-1}_{\bm\theta_0}\|=o(1)$ for $(f, \bm\theta)\in E_n$. Hence, on $E_n$
\begin{eqnarray}
\sqrt{n}(\bm\theta-\bm\theta_0)=\left(\bm J^{-1}_{\bm\theta_0}+o(1)\right)\sqrt{n}(\bm T_{1n}+\bm T_{2n}+\bm T_{3n}),\nonumber
\end{eqnarray}
for sufficiently large $n$, where
\begin{eqnarray}
\bm T_{1n}&=&\int_0^1\left(D_{0,0,1}F(t, f(t),\bm\theta_0)-D_{0,0,1}F(t, f_0(t),\bm\theta_0)\right)^T(f'_0(t)-F(t, f_0(t),\bm\theta_0))w(t)dt,\nonumber\\
\bm T_{2n}&=&\int_0^1\left(D_{0,0,1}F(t, f(t),\bm\theta_0)\right)^T(f'(t)-f'_0(t))w(t)dt,\nonumber\\
\bm T_{3n}&=&\int_0^1\left(D_{0,0,1}F(t, f(t),\bm\theta_0)\right)^T(F(t, f_0(t),\bm\theta_0)-F(t, f(t),\bm\theta_0))w(t)dt.\nonumber
\end{eqnarray}
In view of Lemmas 2 and 4, on a set in the sample space with high true probability, the posterior distribution of $J^{-1}_{\bm\theta_0}\sqrt{n}(\bm T_{1n}+\bm T_{2n}+\bm T_{3n})$ assigns most of its mass inside a large compact set. Thus, we can assert that inside the set $E_n$, the asymptotic behavior of the posterior distribution of $\sqrt{n}(\bm\theta-\bm\theta_0)$ is given by that of
\begin{eqnarray}
{{\bm J^{-1}_{\bm\theta_0}}
\sqrt{n}(\bm T_{1n}+\bm T_{2n}+\bm T_{3n})}.\label{asexpr}
\end{eqnarray}
We shall extract $\sqrt{n}\bm J^{-1}_{\bm\theta_0}\left(\bm\Gamma(f)-\bm\Gamma(f_0)\right)$ from \eqref{asexpr} and show that the remainder term goes to zero. First write
\begin{eqnarray}
\bm T_{2n}&=&-\int_0^1\left(\frac{d}{dt}[(D_{0,0,1}F(t, f_0(t),\bm\theta_0))^Tw(t)]\right)(f(t)-f_0(t))dt\nonumber\\
&&\int_0^1\left(D_{0,0,1}F(t, f(t),\bm\theta_0)-D_{0,0,1}F(t, f_0(t),\bm\theta_0\right)^T(f'(t)-f'_0(t))w(t)dt,\nonumber
\end{eqnarray}
which follows by integration by parts and the fact that $w(0)=w(1)=0$. Note that the first integral of the above equation appears in \eqref{lin}. The norm of the second integral can be bounded above by a constant multiple of $\sup_{t\in[0,1]}|f(t)-f_{0}(t)|^2+\sup_{t\in[0,1]}|f'(t)-f'_{0}(t)|^2$ using the continuity of $D_{0,1,1}F(t,y,\bm\theta)$.
Now we consider $\bm T_{3n}$ in \eqref{asexpr}. Then,
\begin{eqnarray}
\bm T_{3n}&=&\int_0^1\left(D_{0,0,1}F(t, f_0(t),\bm\theta_0)\right)^T(F(t, f_0(t),\bm\theta_0)-F(t, f(t),\bm\theta_0))w(t)dt\nonumber\\
&&+\int_0^1\left(D_{0,0,1}F(t, f(t),\bm\theta_0)-D_{0,0,1}F(t, f_0(t),\bm\theta_0)\right)^T\nonumber\\
&&\times(F(t, f_0(t),\bm\theta_0)-F(t, f(t),\bm\theta_0))w(t)dt.\label{exp3_1}
\end{eqnarray}
The first integral on the right hand side of \eqref{exp3_1} can be written as
\begin{eqnarray}
&-&\int_0^1\left(D_{0,0,1}F(t, f_0(t),\bm\theta_0)\right)^TD_{0,1,0}F(t, f_0(t),\bm\theta_0)(f(t)-f_0(t))w(t)dt\nonumber\\
&&-\int_0^1\left(D_{0,0,1}F(t, f_0(t),\bm\theta_0)\right)^T\nonumber\\
&&\times\left\{\int_0^1[D_{0,1,0}F(t, f_0(t)+\lambda(f-f_0)(t),\bm\theta_0)-D_{0,1,0}F(t, f_0(t),\bm\theta_0)]d\lambda\right\}\nonumber\\
&&\times(f(t)-f_0(t))w(t)dt\nonumber\\
&=&\bm T_{31n}+\bm T_{32n},\nonumber
\end{eqnarray}
say. Now $\bm T_{31n}$ appears in \eqref{lin}. By the continuity of $D_{0,2,0}F(t, y, \bm\theta)$, $\|\bm T_{32n}\|$ can be bounded above up to a constant by a multiple of $\sup_{t\in[0,1]}|f(t)-f_{0}(t)|^2$.
We apply the Cauchy-Schwarz inequality and the continuity of $D_{0,1,1}F(t,y,\bm\theta)$ to bound the second integral on the right hand side of \eqref{exp3_1} by a constant multiple of $\sup\{|f(t)-f_{0}(t)|^2: t\in[0,1]\}$. As far as the first term inside the bracket of \eqref{asexpr} is concerned, we have
\begin{eqnarray}
\bm T_{1n}&=&\int_0^1\left(D_{0,1,0}\bm S(t, f_0(t),\bm\theta_0)\right)(f(t)-f_0(t))w(t)dt\nonumber\\
&&+\int_0^1\left\{\int_0^1\left(D_{0,1,0}\bm S(t, f_0(t)+\lambda(f-f_0)(t),\bm\theta_0)-D_{0,1,0}\bm S(t, f_0(t),\bm\theta_0)\right)d\lambda\right\}\nonumber\\
&&\times(f(t)-f_0(t))w(t)dt\nonumber.
\end{eqnarray}
The first integral appears in \eqref{lin}. The norm of the second integral of the above display can be bounded by a constant multiple of $\sup\{|f(t)-f_{0}(t)|^2: t\in[0,1]\}$ utilizing the continuity of $D_{0,2,1}F(t,y,\bm\theta)$ with respect to its arguments.
Combining these, we find that the norm of $\bm J^{-1}_{\bm\theta_0}\sqrt{n}(\bm T_{1n}+\bm T_{2n}+\bm T_{3n})-\bm J^{-1}_{\bm\theta_0}\sqrt{n}\left(\bm\Gamma(f)-\bm\Gamma(f_0)\right)$ is bounded above by a constant multiple of $\sqrt{n}\sup_{t\in[0,1]}|f(t)-f_{0}(t)|^2+\sqrt{n}\sup_{t\in[0,1]}|f'(t)-f'_{0}(t)|^2$. Now applying Lemma 2, we get the desired result.
\end{proof}

\begin{proof}[Proof of Theorem 2]
By Theorem 1 and \eqref{linearize}, it suffices to show that
\begin{eqnarray}
\left\|\Pi\left(\sqrt{n}\sum_{j=1}^d\bm G_{n,j}^T\bm\beta_j-\sqrt{n}\bm J^{-1}_{\bm\theta_0}\bm\Gamma(\bm f_0)\in\cdot|\bm Y\right)-{N}(\bm\mu_n,\sigma^2\bm\Sigma_n)\right\|_{TV}
=o_{P_0}(1).\nonumber\\\label{thm2_1}
\end{eqnarray}
Note that the posterior distribution of $\bm G^T_{n,j}\bm\beta_j$ is a normal distribution with mean vector $(1+\sigma^2k_n/n)^{-1}\bm G^T_{n,j}(\bm X^T_nX_n)^{-1}\bm X^T_n\bm Y_{,j}$ and dispersion matrix $\sigma^2(1+\sigma^2k_n/n)^{-1}\bm G^T_{n,j}(\bm X^T_nX_n)^{-1}\bm G_{n,j}$. We calculate the Kullback-Leibler divergence between two Gaussian distributions to prove the assertion. Alternatively, we can also follow the approach given in Theorem 1 and Corollary 1 of \citet{bontemps2011bernstein}. The Kulback-Leibler divergence between the distributions $N_{k_n+m-1}(\bm\mu_1,\bm\Omega_1)$ and $N_{k_n+m-1}(\bm\mu_2,\bm\Omega_2)$ is given by
\begin{eqnarray}
\frac{1}{2}\left(\mathrm{tr}(\bm\Omega^{-1}_1\bm\Omega_2)+(\bm\mu_1-\bm\mu_2)^T\bm\Omega^{-1}_1(\bm\mu_1-\bm\mu_2)-(k_n+m-1)-\log(\mathrm{det}(\bm\Omega^{-1}_1\bm\Omega_2))\right).\nonumber
\end{eqnarray}
Taking $\bm\mu_1=(1+\sigma^2k_n/n)^{-1}\bm G^T_{n,j}(\bm X^T_nX_n)^{-1}\bm X^T_n\bm Y_{,j}$, $\bm\Omega_1=\sigma^2(1+\sigma^2k_n/n)^{-1}\bm G^T_{n,j}(\bm X^T_nX_n)^{-1}\bm G_{n,j}$ and $\bm\mu_2=\bm G^T_{n,j}(\bm X^T_nX_n)^{-1}\bm X^T_n\bm Y_{,j}$, $\bm\Omega_2=\sigma^2\bm G^T_{n,j}(\bm X^T_nX_n)^{-1}\bm G_{n,j}$, we get $\mathrm{tr}(\bm\Omega^{-1}_1\bm\Omega_2)=k_n+m-1+o(1)$. Also, $\log(\mathrm{det}(\bm\Omega^{-1}_1\bm\Omega_2))\asymp k_n\log(1+k_n/n)\asymp k^2_n/n=o(1)$. From the proof of Lemma 4, it follows that
\begin{eqnarray}
\lefteqn{(\bm\mu_1-\bm\mu_2)^T\bm\Omega^{-1}_1(\bm\mu_1-\bm\mu_2)}\nonumber\\
&\asymp&n\frac{k^2_n}{n^2}\bm Y^T_{,j}\bm X_n(\bm X^T_nX_n)^{-1}\bm G_{n,j}\bm G^T_{n,j}(\bm X^T_nX_n)^{-1}\bm X^T_n\bm Y_{,j}\nonumber\\
&\lesssim&n\frac{k^2_n}{n^2}\frac{1}{k_n}\frac{k^2_n}{n^2}\frac{n}{k_n}\bm Y^T_{,j}\bm Y_{,j}=o_{P_0}(1).\nonumber
\end{eqnarray}
Hence, the total variation distance between the posterior distribution of $\bm G_{n,j}^T\bm\beta_j$ and a Gaussian distribution with mean $\bm G_{n,j}^T(\bm X^T_n\bm X_n)^{-1}\bm X^T_n\bm Y_{,j}$ and dispersion matrix given by $\sigma^2\bm G_{n,j}^T(\bm X^T_n\bm X_n)^{-1}\bm G_{n,j}$ converges in $P_0-$ probability to zero for $j=1,\ldots,d$.
Since the posterior distributions of $\bm\beta_j$ and $\bm\beta_j'$ are mutually independent for $j\neq j';\,j,j'=1,\ldots,d$, we can assert that the posterior distribution of $\sqrt{n}\sum_{j=1}^d\bm G_{n,j}^T\bm\beta_j-\sqrt{n}\bm J^{-1}_{\bm\theta_0}\bm\Gamma(\bm f_0)$ can be approximated in total variation by $N(\bm\mu_n,\sigma^2\bm\Sigma_n)$.
\end{proof}

\begin{proof}[Proof of Theorem 3]
The marginal posterior of $\sigma^2$ is also inverse gamma with parameters $(dn+2a)/2$ and $b+\sum_{j=1}^d\bm Y^T_{,j}(\bm I_n-\bm P_{\bm X_n}(1+(k_n/n))^{-1})\bm Y_{,j}/2$, where $\bm P_{\bm X_n}=\bm X_n(\bm X^T_n\bm X_n)^{-1}\bm X^T_n$. Straightforward calculations show that
\begin{eqnarray}
\E(\sigma^2|\bm Y)&=&\frac{\frac{1}{2}\sum_{j=1}^d\left\{\bm Y^T_{,j}\bm Y_{,j}-\bm Y^T_{,j}\bm P_{\bm X_n}\bm Y_{,j}(1+k_nn^{-1})^{-1}\right\}+b}{\frac{1}{2}dn+a-1},\nonumber\\
\Var(\sigma^2|\bm Y)&=&\frac{\left(\E(\sigma^2|\bm Y)\right)^2}{\frac{1}{2}dn+a-2},\nonumber
\end{eqnarray}
which give rise to
$|\mathrm E(\sigma^2|\bm Y)-\sigma_0^2|=O_{P_{0}}(n^{-1/2})$ and $\Var(\sigma^2|\bm Y)=O_{P_{0}}(n^{-1})$. In particular, the marginal posterior distribution of $\sigma^2$ is consistent at the true value of error variance. Let $\mathscr{N}$ be an arbitrary neighborhood of $\sigma_0$. Then, $\Pi(\mathscr{N}^c|\bm Y)=o_{P_0}(1)$. We observe that
\begin{eqnarray}
\lefteqn{\sup_{B\in\mathscr{R}^p}\left|\Pi(\sqrt{n}(\bm\theta-\bm\theta_0)\in B|\bm Y)-\Phi(B;\bm\mu_n,\sigma_0^2\bm\Sigma_n)\right|}\nonumber\\
&\leq&\int\sup_{B\in\mathscr{R}^p}\left|\Pi(\sqrt{n}(\bm\theta-\bm\theta_0)\in B|\bm Y,\sigma)-\Phi(B;\bm\mu_n,\sigma^2\bm\Sigma_n)\right|d\Pi(\sigma|\bm Y)\nonumber\\
&&+\int\sup_{B\in\mathscr{R}^p}\left|\Phi(B;\bm\mu_n,\sigma^2\bm\Sigma_n)-\Phi(B;\bm\mu_n,\sigma_0^2\bm\Sigma_n)\right|d\Pi(\sigma|\bm Y)\nonumber\\
&\leq&\sup_{\sigma\in\mathscr{N}}\sup_{B\in\mathscr{R}^p}\left|\Pi(\sqrt{n}(\bm\theta-\bm\theta_0)\in B|\bm Y,\sigma)-\Phi(B;\bm\mu_n,\sigma^2\bm\Sigma_n)\right|\nonumber\\
&&+\sup_{\sigma\in\mathscr{N},B\in\mathscr{R}^p}\left|\Phi(B;\bm\mu_n,\sigma^2\bm\Sigma_n)-\Phi(B;\bm\mu_n,\sigma_0^2\bm\Sigma_n)\right|+2\Pi(\mathscr{N}^c|\bm Y).\nonumber
\end{eqnarray}
The total variation distance between the two normal distributions appearing in the second term is bounded by a constant multiple of $|\sigma-\sigma_0|$, and hence the term can be made arbitrarily small by choosing $\mathscr{N}$ appropriately. The first term converges in probability to zero by \eqref{thm2}. The third term converges in probability to zero by the posterior consistency. Hence, we get the desired result.
\end{proof}

\begin{proof}[Proof of Theorem 4]
According to the fitted model, $\bm Y_{i,}^{1\times d}\sim N_d((\bm X_n)_{i,}\bm B_n,\bm\Sigma^{d\times d})$ for $i=1,\ldots,n$. The logarithm of the posterior probability density function (p.d.f.) is negative half times
\begin{eqnarray}
\sum_{i=1}^n\left((\bm X_n)_{i,}\bm B_n-\bm Y_{i,}\right)\bm\Sigma^{-1}\left(B^T_n(\bm X^T_n)_{,i}-\bm Y^T_{i,}\right)+\sum_{j=1}^d\bm\beta_j^T\frac{\bm X^T_n\bm X_n}{nk_n^{-1}}\bm\beta_j,\label{logpost}
\end{eqnarray}
where $\bm B_n=(\bm\beta_1,\ldots,\bm\beta_d)$. The quadratic term in $\bm\beta_j$ above for $j=1,\ldots,d$, can be consolidated to
\begin{equation}
\mathrm{tr}\left(\left(\bm\Sigma^{-1}+\frac{k_n\bm I_d}{n}\right)\bm B^T_n\bm X^T_n\bm X_n\bm B_n\right).\label{thm3_3}
\end{equation}
The term in \eqref{logpost} which is linear in $\bm\beta_j$, $j=1,\ldots,d$, is given by
\begin{eqnarray}
\sum_{i=1}^n(\bm X_n)_{i,}(\bm\beta_1\ldots\bm\beta_d)\bm\Sigma^{-1}\bm Y^T_{i,}=\mathrm{tr}\left(\bm X_n\bm B_n\bm\Sigma^{-1}\bm Y^T\right)=\mathrm{tr}\left(\bm\Sigma^{-1}\bm Y^T\bm X_n\bm B_n\right).\nonumber
\end{eqnarray}
A completing square argument gives the posterior density to be proportional to
\begin{eqnarray}
\mathrm{exp}\left\{-\frac{1}{2}\mathrm{tr}\left[\left(\bm\Sigma^{-1}+\frac{k_n\bm I_d}{n}\right)\left(\bm B_n-(\bm X^T_n\bm X_n)^{-1}\bm X^T_n\bm Y\bm\Sigma^{-1}\left(\bm\Sigma^{-1}+\frac{k_n\bm I_d}{n}\right)^{-1}\right)^T\right.\right.\nonumber\\
\left.\left.\bm X^T_n\bm X_n\left(\bm B_n-(\bm X^T_n\bm X_n)^{-1}\bm X^T_n\bm Y\bm\Sigma^{-1}\left(\bm\Sigma^{-1}+\frac{k_n\bm I_d}{n}\right)^{-1}\right)\right]\right\},\nonumber
\end{eqnarray}
which can be identified with the pdf of a matrix normal distribution. More precisely,
\begin{eqnarray}
\mathrm{vec}(\bm B_n)|\bm Y&\sim& N\left(\mathrm{vec}\left((\bm X^T_n\bm X_n)^{-1}\bm X^T_n\bm Y\bm\Sigma^{-1}\left(\bm\Sigma^{-1}+\frac{k_n\bm I_d}{n}\right)^{-1}\right),\right.\nonumber\\
&&\left.\left(\bm\Sigma^{-1}+\frac{k_n\bm I_d}{n}\right)^{-1}\otimes(\bm X^T_n\bm X_n)^{-1}\right).\nonumber
\end{eqnarray}
Fixing a $j\in\{1,\ldots,d\}$, we observe that the posterior mean of $\bm\beta_j$ is a weighted sum of $(\bm X^T_n\bm X_n)^{-1}\bm X^T_n\bm Y_{,j'}$ for $j'=1,\ldots,d$. The weight attached with $(\bm X^T_n\bm X_n)^{-1}\bm X^T_n\bm Y_{,j}$ is of the order of $1$, whereas for $j'\neq j$, the contribution from $(\bm X^T_n\bm X_n)^{-1}\bm X^T_n\bm Y_{,j'}$ is of the order of $k_n/n$ which goes to zero as $n$ goes to infinity. Thus, the results of Lemmas 1 to 4 can be shown to hold under this setup.
We are interested in the limiting distribution of $\bm J^{-1}_{\bm\theta_0}\bm\Gamma(\bm f)=\sum_{j=1}^d\bm G_{n,j}^T\bm\beta_j=(\bm G_{n,1}^T\ldots\bm G_{n,d}^T)\mathrm{vec}(\bm B_n)$.
We note that the posterior distribution of $\left(\left(\bm\Sigma^{-1}+k_n\bm I_d/n\right)^{1/2}\otimes\bm I_{k_n+m-1}\right)\mathrm{vec}(\bm B_n)$ is a $(k_n+m-1)d$-dimensional normal distribution with mean vector and dispersion matrix being $\mathrm{vec}\left((\bm X^T_n\bm X_n)^{-1}\bm X^T_n\bm Y\bm\Sigma^{-1}\left(\bm\Sigma^{-1}+k_n\bm I_d/n\right)^{-1/2}\right)$ and $\bm I_d\otimes(\bm X^T_n\bm X_n)^{-1}$ respectively, since by the properties of Kronecker product, for the matrices $\bm A$, $\bm B$ and $\bm D$ of appropriate orders
 $(\bm B^T\otimes\bm A)\mathrm{vec}(\bm D)=\mathrm{vec}(\bm{ADB})$.\\
Let us consider the mean vector of the posterior distribution of the vector $\left(\left(\bm\Sigma^{-1}+k_n\bm I_d/n\right)^{1/2}\otimes\bm I_{k_n+m-1}\right)\mathrm{vec}(\bm B_n)$. We observe that
\begin{eqnarray}
\lefteqn{(\bm X^T_n\bm X_n)^{-1}\bm X^T_n\bm Y\bm\Sigma^{-1}\left(\bm\Sigma^{-1}+\frac{k_n\bm I_d}{n}\right)^{-1/2}}\nonumber\\
&=&(\bm X^T_n\bm X_n)^{-1}\bm X^T_n(\bm Y_{,1}\ldots\bm Y_{,d})\left(\bm \Sigma+\frac{k_n\bm\Sigma^2}{n}\right)^{-1/2}\nonumber\\
&=&(\bm X^T_n\bm X_n)^{-1}\bm X^T_n\left(\sum_{j=1}^d\bm Y_{,j}c_{j1}\ldots\sum_{j=1}^d\bm Y_{,j}c_{jd}\right),\nonumber
\end{eqnarray}
where $\bm C_n=(\!(c_{jk})\!)=\left(\bm \Sigma+k_n\bm\Sigma^2/n\right)^{-1/2}$. For $k=1,\ldots,d$, we define $\bm Z_k$ to be the sub-vector consisting of $[(k-1)(k_n+m-1)+1]^{th}$ to $[k(k_n+m-1)]^{th}$ elements of the vector $\left(\left(\bm\Sigma^{-1}+\frac{k_n\bm I_d}{n}\right)^{1/2}\otimes\bm I_{k_n+m-1}\right)\mathrm{vec}(\bm B_n)$.
Then $\bm Z_k|\bm Y\sim N_{k_n+m-1}\left({(\bm X^T_n\bm X_n)}^{-1}\bm X^T_n\sum_{j=1}^d\bm Y_{,j}c_{jk},{(\bm X^T_n\bm X_n)}^{-1}\right)$. Also, the posterior distributions of $\bm Z_k$ and $\bm Z_{k'}$ are mutually independent for $k\neq k';k,k'=1,\ldots,d$.
Now we will prove that the total variation distance between the posterior distribution of $\bm Z_k$ and $N\left({(\bm X^T_n\bm X_n)}^{-1}\bm X^T_n\sum_{j=1}^d\bm Y_{,j}\sigma^{jk},{(\bm X^T_n\bm X_n)}^{-1}\right)$ converges in $P_0-$probability to zero for $k=1,\ldots,d$, where $\bm\Sigma^{-1/2}=(\!(\sigma^{jk})\!)$.  The total variation distance between two multivariate normal distributions with equal dispersion matrix $(\bm X^T_n\bm X_n)^{-1}$ and mean vectors ${(\bm X^T_n\bm X_n)}^{-1}\bm X^T_n\sum_{j=1}^d\bm Y_{,j}c_{jk}$ and ${(\bm X^T_n\bm X_n)}^{-1}\bm X^T_n\sum_{j=1}^d\bm Y_{,j}\sigma^{jk}$ is bounded by $\sum_{j=1}^d\|{(\bm X^T_n\bm X_n)}^{-1/2}\bm X^T_n\bm Y_{,j}(c_{jk}-\sigma^{jk})\|$.
Fixing $k$, for $j=1,\ldots,d$, we have that
\begin{eqnarray}
\|{(\bm X^T_n\bm X_n)}^{-1/2}\bm X^T_n\bm Y_{,j}(c_{jk}-\sigma^{jk})\|&=&|c_{jk}-\sigma^{jk}|\left(\bm Y_{,j}^T\bm X_n{(\bm X^T_n\bm X_n)}^{-1}\bm X^T_n\bm Y_{,j}\right)^{1/2}\nonumber\\
&\leq&|c_{jk}-\sigma^{jk}|\left(\bm Y_{,j}^T\bm Y_{,j}\right),\nonumber
\end{eqnarray}
since the eigenvalues of $\bm X_n{(\bm X^T_n\bm X_n)}^{-1}\bm X^T_n$ are either zero or $1$. Since clearly $\bm C_n$ converges to $\bm\Sigma^{-1/2}$ at the rate $k_n/n$, we have for $j=1,\ldots,d$,
\begin{eqnarray}
\|{(\bm X^T_n\bm X_n)}^{-1}\bm X^T_n\bm Y_{,j}(c_{jk}-\sigma^{jk})\|\lesssim\frac{k_n}{n}O_{P_0}(\sqrt{n})=o_{P_0}(1).
\end{eqnarray}
Hence, we conclude that the total variation distance between the distributions $N\left({(\bm X^T_n\bm X_n)}^{-1}\bm X^T_n\sum_{j=1}^d\bm Y_{,j}c_{jk},{(\bm X^T_n\bm X_n)}^{-1}\right)$ and $N\left({(\bm X^T_n\bm X_n)}^{-1}\bm X^T_n\sum_{j=1}^d\bm Y_{,j}\sigma^{jk},{(\bm X^T_n\bm X_n)}^{-1}\right)$ converges to zero in $P_0-$probability.
Note that we can write $\left(\bm G_{n,1}^T\ldots\bm G_{n,d}^T\right)\mathrm{vec}(\bm B_n)$ in terms of $\bm Z_k$ as
\begin{eqnarray}
\sum_{k=1}^d\mathrm{cols}_{(k-1)(k_n+m-1)+1}^{k(k_n+m-1)}\left(\left(\bm G_{n,1}^T\ldots\bm G_{n,d}^T\right)\left(\left(\bm\Sigma^{-1}+\frac{k_n\bm I_d}{n}\right)^{1/2}\otimes\bm I_{k_n+m-1}\right)^{-1}\right)\times\bm Z_k.\nonumber
\end{eqnarray}
Since the posterior distributions of $\bm Z_k$, $k=1,\ldots,d$ are independent, we therefore obtain
\begin{eqnarray}
\left\|\left(\sqrt{n}\left(\bm G_{n,1}^T\ldots\bm G_{n,d}^T\right)\mathrm{vec}(\bm B_n)-\sqrt{n}\bm J^{-1}_{\bm\theta_0}(\bm f_0)\right)-N(\bm\mu_n^{**},\bm\Sigma_n^{**})\right\|_{TV}=o_{P_0}(1),\nonumber
\end{eqnarray}
where $\bm\mu^{**}_n$ is given by
\begin{eqnarray}
&&\sqrt{n}\sum_{k=1}^d\mathrm{cols}_{(k-1)(k_n+m-1)+1}^{k(k_n+m-1)}\left(\left(\bm G^T_{n,1}\ldots\bm G^T_{n,d}\right)\left(\left(\bm\Sigma^{-1}+\frac{k_n\bm I_d}{n}\right)^{1/2}\otimes\bm I_{k_n+m-1}\right)^{-1}\right)\nonumber\\
&&\times{(\bm X^T_n\bm X_n)}^{-1}\bm X_n^T\sum_{j=1}^d\bm Y_{,j}\sigma^{jk}
-\bm J_{\bm\theta_0}^{-1}\sqrt{n}\bm\Gamma(\bm f_0),\nonumber
\end{eqnarray}
and $\bm\Sigma^{**}_n$ is given by
\begin{eqnarray}
\lefteqn{n\sum_{k=1}^d\mathrm{cols}_{(k-1)(k_n+m-1)+1}^{k(k_n+m-1)}\left(\left(\bm G^T_{n,1}\ldots\bm G^T_{n,d}\right)\left(\left(\bm\Sigma^{-1}+\frac{k_n\bm I_d}{n}\right)^{1/2}\otimes\bm I_{k_n+m-1}\right)^{-1}\right)}\nonumber\\
&&\times{(\bm X^T_n\bm X_n)}^{-1}\nonumber\\
&&\times\mathrm{rows}_{(k-1)(k_n+m-1)+1}^{k(k_n+m-1)}\left(\left(\left(\bm\Sigma^{-1}+\frac{k_n\bm I_d}{n}\right)^{1/2}\otimes\bm I_{k_n+m-1}\right)^{-1}\left(\bm G^T_{n,1}\ldots\bm G^T_{n,d}\right)^T\right).\nonumber
\end{eqnarray}
Following the steps of the proof of Lemma 4, it can be shown that the eigenvalues of the matrix $\bm\Sigma^*_n$ mentioned in the statement of Theorem 4 are bounded away from zero and infinity. We can show that the Kullback-Leibler divergence of $N(\bm\mu^{**}_n,\bm\Sigma^{**}_n)$ from $N(\bm\mu^{*}_n,\sigma^2\bm\Sigma^{*}_n)$ converges in probability to zero by going through some routine matrix manipulations. Hence,
\begin{eqnarray}
\left\|\left(\sqrt{n}\left(\bm G_{n,1}^T\ldots\bm G_{n,d}^T\right)\mathrm{vec}(\bm B_n)-\sqrt{n}\bm J^{-1}_{\bm\theta_0}(\bm f_0)\right)-N(\bm\mu_n^{*},\sigma^2\bm\Sigma_n^{*})\right\|_{TV}=o_{P_0}(1).\nonumber
\end{eqnarray}
The above expression is equivalent to \eqref{thm2_1} of the proof of Theorem 2. Following steps similar to those of Theorem 2, we get \eqref{thm4_1}. We obtain \eqref{thm4_2} by following the proof of Theorem 3.
\end{proof}

\section{Appendix}

We need to go through a couple of lemmas in order to prove the main results. We denote by $\E_0(\cdot)$ and ${\Var}_0(\cdot)$ the expectation and variance operators respectively with respect to $P_0-$ probability. The following lemma helps to estimate the bias of the Bayes estimator.
\begin{lemma}
For $m\geq2$ and $k_n$ satisfying $n^{1/2m}\ll k_n\ll n$, for $r=0,1$, $\sup_{t\in[0,1]}|\E_0(\E(f^{(r)}(t)|\bm Y))-f^{(r)}_{0}(t)|=o(k_n^{r+1/2}/\sqrt{n})$.
\end{lemma}
\begin{proof}
We note that $f^{(r)}(t)=(\bm N^{(r)}(t))^T\bm\beta$ for $r=0,1$ with $\bm N^{(r)}(\cdot)$ standing for the $r^{th}$ order derivative of $\bm N(\cdot)$. By \eqref{posterior},
\begin{equation}
\E(f^{(r)}(t)|\bm Y)={\left(1+\frac{k_n\sigma^2}{n}\right)}^{-1}(\bm N^{(r)}(t))^T{({\bm X^T_n\bm X_n})}^{-1}\bm X^T_n\bm Y.\label{postexp}
\end{equation}
\citet{zhou2000derivative} showed that
\begin{eqnarray}
(\bm N^{(r)}(t))^T{({\bm X^T_n\bm X_n})}^{-1}\bm N^{(r)}(t)
\asymp\frac{k_n^{2r+1}}{n}.\label{var}
\end{eqnarray}
Since $ f^{(r)}_0\in C^{(m-r)}$, there exists a $\bm\beta^*$ \citep[Theorem XII.4, page 178]{de1978practical} such that
\begin{equation}
\sup_{t\in[0,1]}|f^{(r)}_{0}(t)-(\bm N^{(r)}(t))^T\bm\beta^*|=O({k_n^{-(m-r)}}).\label{spldis}
\end{equation}
For any $t\in[0,1]$, we can bound the absolute bias of $\E(f^{(r)}_0(t)|\bm Y)$ multiplied with $\sqrt{n}k_n^{-r-1/2}$ by %\begin{small}
\begin{eqnarray}
\lefteqn{\sqrt{n}k_n^{-r-1/2}\sup_{t\in[0,1]}|\E_0(\E(f^{(r)}(t)|\bm{Y}))-f^{(r)}_{0}(t)|}\nonumber\\
&\leq&\sqrt{n}k_n^{-r-1/2}\sup_{t\in[0,1]}\left|{\left(1+\frac{k_n\sigma^2}{n}\right)}^{-1}(\bm N^{(r)}(t))^T\bm\beta^*-(\bm N^{(r)}(t))^T\bm\beta^*\right|\nonumber\\
&&+\sqrt{n}k_n^{-r-1/2}{\left(1+\frac{k_n\sigma^2}{n}\right)}^{-1}\sup_{t\in[0,1]}|(\bm N^{(r)}(t))^T{({\bm X^T_n\bm X_n})}^{-1}\bm X^T_n(f_{0}(\bm{x})-\bm X_n\bm\beta^*)|\nonumber\\
&&+\sqrt{n}k_n^{-r-1/2}\sup_{t\in[0,1]}|f^{(r)}_{0}(t)-(\bm N^{(r)}(t))^T\bm\beta^*|.\nonumber
\end{eqnarray}
%\end{small}
Using the fact that $\sup_{t\in[0,1]}|(\bm N^{(r)}(t))^T\bm\beta^*|=O(1)$, first term on the right hand side of the previous inequality is of the order of $k_n^{-r+1/2}/\sqrt{n}$. Using the Cauchy-Schwarz inequality, \eqref{var} and \eqref{spldis}, we can bound the second term up to a constant multiple by $\sqrt{n}k_n^{-m}$. The third term has the order of $\sqrt{n}k_n^{-m-1/2}$ as a result of \eqref{spldis}.
By the assumed conditions on $m$ and $k_n$, the assertion holds.
\end{proof}

The following lemma controls posterior variability.
\begin{lemma}
If $m\geq5$ and $n^{1/2m}\ll k_n\ll n^{1/8}$, then for $r=0,1$ and for all $\epsilon>0$, $\Pi\left(\sqrt{n}\sup_{t\in[0,1]}|f^{(r)}(t)-f^{(r)}_{0}(t)|^2>\epsilon|\bm Y\right)=o_{P_0}(1)$.
\end{lemma}
\begin{proof}
By Markov's inequality and the fact that $|a+b|^2\leq 2(|a|^2+|b|^2)$ for two real numbers $a$ and $b$, we can bound
$\Pi\left(\sup_{t\in[0,1]}\sqrt{n}|f^{(r)}(t)-f^{(r)}_{0}(t)|^2>\epsilon|\bm Y\right)$ by
\begin{eqnarray}
\lefteqn{2\frac{\sqrt{n}}{\epsilon}\left\{\sup_{t\in[0,1]}\left|\E(f^{(r)}(t)|\bm Y)-f^{(r)}_{0}(t)\right|^2\right\}}\nonumber\\
&&\left.+\E\left[\sup_{t\in[0,1]}\left|f^{(r)}(t)-\E(f^{(r)}(t)|\bm Y)\right|^2\large{|}\bm Y\right]\right\}.\label{chv}
\end{eqnarray}
Now we obtain the asymptotic orders of the expectations of the two terms inside the bracket above. We can bound the expectation of the first term by
\begin{eqnarray}
\lefteqn{2\sup_{t\in[0,1]}\left|\E_0(\E(f^{(r)}(t)|\bm Y))-f^{(r)}_{0}(t)\right|^2}\nonumber\\
&&+2\E_0\left[\sup_{t\in[0,1]}\left|\E(f^{(r)}(t)|\bm Y)-\E_0(\E[f^{(r)}(t)|\bm Y])\right|^2\right].\label{chv_1}
\end{eqnarray}
Using \eqref{postexp}, $\sup_{t\in[0,1]}\left|\E(f^{(r)}(t)|\bm Y)-\E_0(\E[f^{(r)}(t)|\bm Y])\right|$ can be bounded up to a constant multiple by
\begin{eqnarray}
\lefteqn{\max_{1\leq k\leq n}\left|(\bm N^{(r)}(s_k))^T{({\bm X^T_n\bm X_n})}^{-1}\bm X^T_n\bm\varepsilon\right|}\nonumber\\
&&+\sup_{t,t':|t-t'|\leq n^{-1}}\left|(\bm N^{(r)}(t)-\bm N^{(r)}(t'))^T{({\bm X^T_n\bm X_n})}^{-1}\bm X^T_n\bm\varepsilon\right|\nonumber,
\end{eqnarray}
where $s_k=k/n$ for $k=1,\ldots,n$. Applying the mean value theorem to the second term of the above sum, we can bound the expression inside the $\E_0$ expectation in the second term of \eqref{chv_1} by a constant multiple of
\begin{eqnarray}
\lefteqn{\max_{1\leq k\leq n}\left|(\bm N^{(r)}(s_k))^T{({\bm X^T_n\bm X_n})}^{-1}\bm X^T_n\bm\varepsilon\right|^2}\nonumber\\
&&+\sup_{t\in[0,1]}\frac{1}{n^2}\left|(\bm N^{(r+1)}(t))^T{({\bm X^T_n\bm X_n})}^{-1}\bm X^T_n\bm\varepsilon\right|^2.\label{A1}
\end{eqnarray}
By the spectral decomposition, we can write $\bm X_n{({\bm X^T_n\bm X_n})}^{-1}\bm X^T_n=\bm P^T\bm D\bm P$, where $\bm P$ is an orthogonal matrix and $\bm D$ is a diagonal matrix with $k_n+m-1$ ones and $n-k_n-m+1$ zeros in the diagonal. Now using the Cauchy-Schwarz inequality, we get
\begin{eqnarray}
\lefteqn{\E_0\left(\max_{1\leq k\leq n}\left|(\bm N^{(r)}(s_k))^T{({\bm X^T_n\bm X_n})}^{-1}\bm X^T_n\bm\varepsilon\right|^2\right)}\nonumber\\
&\leq&\max_{1\leq k\leq n}\left\{(\bm N^{(r)}(s_k))^T{({\bm X^T_n\bm X_n})}^{-1}\bm N^{(r)}(s_k)\right\}
\E_0\left(\bm\varepsilon^T\bm P^T\bm D\bm P\bm\varepsilon\right).\nonumber
\end{eqnarray}
By Lemma 5.4 of \citet{zhou2000derivative} and the fact that $\bm\Var_0(\bm P\bm\varepsilon)=\bm\Var_0(\bm\varepsilon)$, we can conclude that the expectation of the first term of \eqref{A1} is $O(k_n^{2r+2}/n)$. Again applying the Cauchy-Schwarz inequality, the second term of \eqref{A1} is bounded by
\begin{eqnarray}
\sup_{t\in[0,1]}\left\{\frac{1}{n^2}(\bm N^{(r+1)}(t))^T{(\bm X_n^T\bm X_n)}^{-1} \bm N^{(r+1)}(t)\right\}(\bm\varepsilon^T\bm\varepsilon),\nonumber
\end{eqnarray}
whose expectation is of the order $n(k_n^{2r+3}/n^3)=k_n^{2r+3}/n^2$, using Lemma 5.4 of Zhou and Wolfe (2000).
Thus, the expectation of the bound given by \eqref{A1} is of the order $k_n^{2r+2}/n$. Combining it with \eqref{chv_1} and Lemma 1, we get
\begin{equation}
\E_0\left[\sup_{t\in[0,1]}\left|\E(f^{(r)}(t)|\bm Y)-f^{(r)}_{0}(t)\right|^2\right]=O\left(\frac{k_n^{2r+2}}{n}\right).\label{chv_2}
\end{equation}
Define $\bm \varepsilon^*:={(\bm X^T_n\bm X_n)}^{1/2}\bm\beta-{\left(1+\frac{\sigma^2k_n}{n}\right)}^{-1}{(\bm X^T_n\bm X_n)}^{-1/2}\bm X^T_n\bm Y$. Note that $\bm\varepsilon^*|\bm Y\sim N(\bm 0, {\left({\sigma^{-2}}+{k_n/n}\right)}^{-1}\bm I_{k_n+m-1})$. Expressing $\sup_{t\in[0,1]}|f^{(r)}(t)-\E[f^{(r)}(t)|\bm Y]|$ as $\sup_{t\in[0,1]}\left|(\bm N^{(r)}(t))^T{({\bm X^T_n\bm X_n})}^{-1/2}\bm\varepsilon^*\right|$ and using the Cauchy-Schwarz inequality and Lemma 5.4 of \citet{zhou2000derivative}, the second term inside the bracket in \eqref{chv} is seen to be $O(k_n^{2r+2}/n)$. Combining it with \eqref{chv} and \eqref{chv_2} and using $2r+2\leq4$, we have the desired assertion.
\end{proof}

Lemmas 1 and 2 can be used to prove the posterior consistency of $\bm\theta$ as shown in the next lemma.
\begin{lemma}
If $m\geq5$ and $n^{1/2m}\ll k_n\ll n^{1/8}$, then for all $\epsilon>0$, $\Pi(\|\bm\theta-\bm\theta_0\|>\epsilon|\bm Y)=o_{P_0}(1).$
\end{lemma}
\begin{proof}
By the triangle inequality, using the definition in \eqref{argmin},
\begin{eqnarray}
%\lefteqn{|R_{ f}(\bm\eta)-R_{ f_0}(\bm\eta)|}\nonumber\\
|R_{ f}(\bm\eta)-R_{ f_0}(\bm\eta)|&\leq&\left\|f'(\cdot)-f'_0(\cdot)\right\|_w+\left\|F(\cdot, f(\cdot), \bm\eta)-F(\cdot, f_0(\cdot), \bm\eta)\right\|_w\nonumber\\
&\leq&c_1\sup_{t\in[0,1]}|f'(t)-f'_{0}(t)|+c_2\sup_{t\in[0,1]}|f(t)-f_{0}(t)|,\nonumber
\end{eqnarray}
for appropriately chosen constants $c_1$ and $c_2$. We denote the set $T_n=\{f:\sup_{t\in[0,1]}|f(t)-f_0(t)|\leq\tau_n,\,\sup_{t\in[0,1]}|f'(t)-f'_0(t)|\leq\tau_n\}$ for some $\tau_n\rightarrow0$. By Lemma 2, there exists such a sequence $\{\tau_n\}$ so that $\Pi(T^c_n|\bm Y)=o_{P_0}(1)$. Hence for $f\in T_n$,
\begin{eqnarray}
\sup_{\bm\eta\in\bm\Theta}|R_{f}(\bm\eta)-R_{f_0}(\bm\eta)|\leq (c_1+c_2)\tau_n=o(1)\nonumber
\end{eqnarray}
Therefore, for any $\delta>0$, $\Pi(\sup_{\bm\eta\in\bm\Theta}|R_{f}(\bm\eta)-R_{f_0}(\bm\eta)|>\delta|\bm Y)=o_{P_0}(1)$. By assumption \eqref{assmp}, for $\|\bm\theta-\bm\theta_0\|\geq\epsilon$ there exists a $\delta>0$ such that
\begin{eqnarray}
\delta<R_{f_0}(\bm\theta)-R_{f_0}(\bm\theta_0)
&\leq&R_{f_0}(\bm\theta)-R_{f}(\bm\theta)+R_{f}(\bm\theta_0)-R_{f_0}(\bm\theta_0)\nonumber\\
&\leq&2\sup_{\bm\eta\in\bm\Theta}|R_{f}(\bm\eta)-R_{f_0}(\bm\eta)|,\nonumber
\end{eqnarray}
since $R_{f}(\bm\theta)\leq R_{f}(\bm\theta_0)$. Consequently,
\begin{eqnarray}
\Pi(\|\bm\theta-\bm\theta_0\|>\epsilon|\bm Y)\leq\Pi\left(\sup_{\bm\eta\in\bm\Theta}|R_{f}(\bm\eta)-R_{f_0}(\bm\eta)|>{\delta}/{2}|\bm Y\right)=o_{P_0}(1).\nonumber
\end{eqnarray}
\end{proof}

The asymptotic behavior of the posterior distribution of $\sqrt{n}\bm J^{-1}_{\bm\theta_0}(\bm\Gamma(\bm f)-\bm\Gamma(\bm f_0))$ is given by the next lemma.
\begin{lemma}
Under the conditions of Theorem 2, on a set in the sample space with high true probability, the posterior distribution of $\sqrt{n}\bm J^{-1}_{\bm\theta_0}(\bm\Gamma(\bm f)-\bm\Gamma(\bm f_0))$ assigns most of its mass inside a large compact set.
\end{lemma}
\begin{proof}
First note that $\bm J^{-1}_{\bm\theta_0}\bm\Gamma(\bm f)=\sum_{j=1}^d\bm G^T_{n,j}\bm\beta_j$ and $\bm J^{-1}_{\bm\theta_0}\bm\Gamma(\bm f_0)=\sum_{j=1}^d\int_0^1\bm A_{,j}(t)f_{j0}(t)dt$, where $\bm A_{,j}(t)$ denotes the $j^{th}$ column of the matrix $\bm A(t)$ as defined in Remark $4$ for $j=1,\ldots,d$. In order to prove the assertion, we will show that $\bm\Var(\bm G_{n,j}^T\bm\beta_j|\bm Y)$ and $\bm\Var_0(\bm\E(\bm G_{n,j}^T\bm\beta_j|\bm Y))$ have all eigenvalues of the order $n^{-1}$ and
\begin{eqnarray}
\max_{1\leq k\leq p}\left|[\bm\E_0(\bm\E(\bm G_{n,j}^T\bm\beta_j|\bm Y))]_k-\int_0^1 A_{k,j}(t)f_{j0}(t)dt\right|&=&o\left({n}^{-1/2}\right),\nonumber
\end{eqnarray}
for $k=1,\ldots,p$, $j=1,\ldots,d$, where $A_{k,j}(t)$ is the $(k,j)^{th}$ element of the matrix $\bm A(t)$.
Let us fix $j\in\{1,\ldots,d\}$. We note that
\begin{eqnarray}
\bm\E(\bm G_{n,j}^T\bm\beta_j|\bm Y)={\left(1+\frac{k_n\sigma^2}{n}\right)}^{-1}\bm G_{n,j}^T{({\bm X^T_n\bm X_n})}^{-1}\bm X^T_n\bm Y_{,j}.\nonumber
\end{eqnarray}
Hence,
\begin{eqnarray}
\bm{\Var}_0(\bm\E(\bm G_{n,j}^T\bm\beta_j|\bm Y))=\sigma_0^2{\left(1+\frac{\sigma^2k_n}{n}\right)}^{-2}\bm G_{n,j}^T{({\bm X^T_n\bm X_n})}^{-1}\bm G_{n,j}\nonumber.
\end{eqnarray}
Also note that
\begin{eqnarray}
\bm\Var(\bm G_{n,j}^T\bm\beta_j|\bm Y)=\sigma^2{\left(1+\frac{\sigma^2k_n}{n}\right)}^{-1}\bm G_{n,j}^T{({\bm X^T_n\bm X_n})}^{-1}\bm G_{n,j}\nonumber.
\end{eqnarray}
If $A_{k,j}(\cdot)\in C^{m^*}\left((0,1)\right)$ for some $1\leq m^*<m$, then by equation $(2)$ of \citet[page 167]{de1978practical}, we have $\sup\{|A_{k,j}(t)-\tilde A_{k,j}(t)|:t\in[0,1]\}=O(k_n^{-1})$, where $\tilde A_{k,j}(\cdot)=\bm \alpha_{k,j}^T\bm N(\cdot)$ and $\bm\alpha^T_{k,j}=(A_{k,j}(t^*_1),\ldots,A_{k,j}(t^*_{k_n+m-1}))$ with appropriately chosen $t^*_1,\ldots,t^*_{k_n+m-1}$.
We can express $\bm G^T_{n,j}(\bm X^T_n\bm X_n)^{-1}\bm G_{n,j}$ as
\begin{eqnarray}
&&(\bm G_{n,j}-\tilde{\bm G}_{n,j})^T(\bm X^T_n\bm X_n)^{-1}(\bm G_{n,j}-\tilde{\bm G}_{n,j})+\tilde{\bm G}^T_{n,j}(\bm X^T_n\bm X_n)^{-1}(\bm G_{n,j}-\tilde{\bm G}_{n,j})\nonumber\\
&&+\tilde{\bm G}^T_{n,j}(\bm X^T_n\bm X_n)^{-1}\tilde{\bm G}_{n,j}+(\bm G_{n,j}-\tilde{\bm G}_{n,j})^T(\bm X^T_n\bm X_n)^{-1}\tilde{\bm G}_{n,j}\nonumber
\end{eqnarray}
where $[\tilde{\bm G}^T_{n,j}]_{k,}=\int_0^1\tilde{A}_{k,j}(t)(\bm N(t))^Tdt$ for $k=1,\ldots,p.$ Let $\tilde{\bm A}=(\!(\tilde A_{k,j})\!)$. We study the asymptotic orders of the eigenvalues of the matrices $\tilde{\bm G}^T_{n,j}(\bm X^T_n\bm X_n)^{-1}\tilde{\bm G}_{n,j}$ and $(\bm G_{n,j}-\tilde{\bm G}_{n,j})^T(\bm X^T_n\bm X_n)^{-1}(\bm G_{n,j}-\tilde{\bm G}_{n,j})$. Note that
\begin{eqnarray}
\bm\alpha^T_{k,j}\int_0^1\bm N(t)\bm N^T(t)dt\,\bm\alpha_{k,j}=\int_0^1\tilde A^2_{k,j}(t)dt
\asymp\|\bm\alpha_{k,j}\|^2k^{-1}_n,\nonumber
\end{eqnarray}
by Lemma 6.1 of \citet{zhou1998local} implying that the eigenvalues of the matrix $\int_0^1\bm N(t)(\bm N(t))^Tdt$ are of order $k^{-1}_n$. Since the eigenvalues of $\left(\bm X^T_n\bm X_n/n\right)$ are of the order $k_n^{-1}$ \citep{zhou1998local}, we have
\begin{eqnarray}
\lefteqn{\mathrm{maxeig}\left(\tilde{\bm G}^T_{n,j}(\bm X^T_n\bm X_n)^{-1}\tilde{\bm G}_{n,j}\right)}\nonumber\\
&\lesssim&\frac{k_n}{n}\mathrm{maxeig}\left(\tilde{\bm G}^T_{n,j}\tilde{\bm G}_{n,j}\right)\nonumber\\
&=&\frac{k_n}{n}\mathrm{maxeig}\left(\int_0^1\tilde{\bm A_{,j}}(t)\bm N^T(t)dt\int_0^1\bm N(t)(\tilde{\bm A_{,j}}(t))^Tdt\right)\nonumber\\
&=&\frac{k_n}{n}\mathrm{maxeig}\left(\left(\bm\alpha_{1,j}\cdots\bm\alpha_{p,j}\right)^T\left(\int_0^1\bm N(t)\bm N^T(t)dt\right)^2\left(\bm\alpha_{1,j}\cdots\bm\alpha_{p,j}\right)\right)\nonumber\\
&\lesssim&\frac{1}{nk_n}\mathrm{maxeig}(\!(\bm\alpha^T_{k,j}\bm\alpha_{l,j})\!)\nonumber\\
&\asymp&\frac{1}{n}\mathrm{maxeig}(\!(\langle A_{k,j}(\cdot),A_{l,j}(\cdot)\rangle)\!)\nonumber\\
&=&\frac{1}{n}\mathrm{maxeig}\left(\bm B_j\right)\asymp\frac{1}{n}.\nonumber
\end{eqnarray}
Similarly, it can be shown that $\mathrm{mineig}\left(\tilde{\bm G}^T_{n,j}(\bm X^T_n\bm X_n)^{-1}\tilde{\bm G}_{n,j}\right)\gtrsim n^{-1}$. Let us denote by $\bm 1_{k_n+m-1}$ the $k_n+m-1$-component vector with all elements $1$. Then for $k=1,\ldots,p$,
\begin{eqnarray}
\lefteqn{\left[(\bm G_{n,j}-\tilde{\bm G}_{n,j})^T(\bm X^T_n\bm X_n)^{-1}(\bm G_{n,j}-\tilde{\bm G}_{n,j})\right]_{k,k}}\nonumber\\
&=&\int_0^1(A_{k,j}(t)-\tilde A_{k,j}(t))(\bm N(t))^Tdt\,{({\bm X^T_n\bm X_n})}^{-1}\int_0^1(A_{k,j}(t)-\tilde A_{k,j}(t))(\bm N(t))dt\nonumber\\
&=&\frac{1}{n}\int_0^1(A_{k,j}(t)-\tilde A_{k,j}(t))(\bm N(t))^Tdt\,{\left({\bm X^T_n\bm X_n}/{n}\right)}^{-1}\int_0^1(A_{k,j}(t)-\tilde A_{kj}(t))\bm N(t)dt\nonumber\\
&\asymp&\frac{k_n}{n}\int_0^1(A_{k,j}(t)-\tilde A_{k,j}(t))(\bm N(t))^Tdt\int_0^1(A_{k,j}(t)-\tilde A_{k,j}(t))\bm N(t)dt\nonumber\\
&\lesssim&\frac{1}{nk_n},\nonumber
\end{eqnarray}
the last step following by the application of the Cauchy-Schwarz inequality and the facts that $\sup\{|A_{k,j}(t)-\tilde A_{k,j}(t)|:t\in[0,1]\}=O(k_n^{-1})$ and $\int_0^1\|\bm N(t)\|^2dt\leq 1$. Thus, the eigenvalues of $(\bm G_{n,j}-\tilde{\bm G}_{n,j})^T(\bm X^T_n\bm X_n)^{-1}(\bm G_{n,j}-\tilde{\bm G}_{n,j})$ are of the order $(nk_n)^{-1}$ or less. Hence, the eigenvalues of $\bm G^T_{n,j}(\bm X^T_n\bm X_n)^{-1}\bm G_{n,j}$ are of the order $n^{-1}$.\\
Similar to the proof of Lemma 1, we can write for the $\bm\beta^*_j$ given in \eqref{spldis},
\begin{eqnarray}
\lefteqn{\sqrt{n}\left|[\bm\E_0(\bm\E(\bm G_{n,j}^T\bm\beta_j|\bm{Y}))]_k-\int_0^1 A_{k,j}(t)f_{j0}(t)dt\right|}\nonumber\\
&\leq&\sqrt{n}\left|{\left(1+\frac{k_n\sigma^2}{n}\right)}^{-1}[\bm G_{n,j}^T\bm\beta^*_j]_k-[\bm G_{n,j}^T\bm\beta^*_j]_k\right|\nonumber\\
&&+\sqrt{n}{\left(1+\frac{k_n\sigma^2}{n}\right)}^{-1}\left|[\bm G_{n,j}^T{(\bm X^T_n\bm X_n)}^{-1}\bm X^T_n(f_{j0}(\bm{x})-\bm X_n\bm\beta^*_j)]_k\right|\nonumber\\
&&+\sqrt{n}\left|\int_0^1 A_{k,j}(t)f_{j0}(t)dt-[\bm G_{n,j}^T\bm\beta^*_j]_k\right|,\nonumber
\end{eqnarray}
where $[\bm G_{n,j}^T\bm\beta^*_j]_k=\int_0^1A_{k,j}(t)f^*_j(t)dt$ and $f^*_j(t)=\bm N^T(t)\bm\beta^*_j$ for $k=1,\ldots,p$. Proceeding in the same way as in the proof of Lemma 1, we can show that each term on the right hand side of the above equation converges to zero. Hence, the proof.
\end{proof}

\bibliographystyle{chicago}
\bibliography{ref}

\end{document}